\documentclass[a4paper,reqno]{amsart}
\usepackage{amsmath}
\usepackage{amssymb}
\usepackage{amsthm}
\usepackage{latexsym}
\usepackage{amscd}
\usepackage{xypic}
\xyoption{curve}
\usepackage{ifthen} 
\usepackage{hyperref} 
\usepackage{graphicx}
\usepackage{enumerate}
\usepackage{enumitem}
\usepackage{rotating}
\usepackage{xcolor}
\usepackage{mathtools}
\usepackage{todonotes}
\usepackage{tikz-cd}

\usepackage{float}
\restylefloat{table}

\numberwithin{equation}{section}

\def\cocoa{{\hbox{\rm C\kern-.13em o\kern-.07em C\kern-.13em o\kern-.15em A}}}

\newtheorem{theorem}{Theorem}[section]

\newtheorem{lemma}[theorem]{Lemma}
\newtheorem{proposition}[theorem]{Proposition}
\newtheorem{corollary}[theorem]{Corollary}

\theoremstyle{definition}
\newtheorem{remark}[theorem]{Remark}
\newtheorem{definition}[theorem]{Definition}

\newcommand {\sHom}{\mathcal{H}\kern -0.25ex{\mathit om}}
\newcommand {\sExt}{\mathcal{E}\kern -0.25ex{\mathit xt}}
\newcommand {\sTor}{\mathcal{T}\kern -0.25ex{\mathit or}}

\newcommand {\Ext}{\mathrm{Ext}}

\newcommand {\Hilb}{\mathcal{H}\kern -0.25ex{\mathit ilb\/}}

\newcommand {\rh}{{h}}

\newcommand{\cC}{{\mathcal C}}

\newcommand{\cE}{{\mathcal E}}
\newcommand{\cEv}{{\mathcal E}^{\vee}}
\newcommand{\cF}{{\mathcal F}}
\newcommand{\cM}{{\mathcal M}}

\newcommand{\cO}{{\mathcal O}}

\newcommand{\cI}{{\mathcal I}}

\newcommand {\bZ}{\mathbb{Z}}

\newcommand {\bC}{\mathbb{C}}
\newcommand {\bP}{\mathbb{P}}

\DeclareRobustCommand{\rchi}{{\mathpalette\irchi\relax}}
\newcommand{\irchi}[2]{\raisebox{\depth}{$#1\chi$}}

\newcommand{\p}{\mathbb{P}}

\newcommand{\ls}{\mathcal{L}}
\newcommand{\oo}{\mathcal{O}}

\def\p#1{{\bP^{#1}}}
\def\H#1{\mathrm{H}^{#1}}
\def\h#1{\mathrm{h}^{#1}}
\def\ext#1{\mathrm{ext}^{#1}}

\newcommand*{\triple}[2][.1ex]{%
	\mathrel{\vcenter{\offinterlineskip%
			\hbox{$#2$}\vskip#1\hbox{$#2$}\vskip#1\hbox{$#2$}}}}
\newcommand*{\triplerightarrow}{\triple{\rightarrow}}

\def\p#1{{\bP^{#1}}}
\def\H#1{\mathrm{H}^{#1}}
\def\h#1{\mathrm{h}^{#1}}

\title[$\ell$-away ACM Bundles]{$\ell$-away ACM Bundles on Fano Surfaces}

\thanks{The second author is supported by the grant MAESTRO NCN - UMO-2019/34/A/ST1/00263 - Research in Commutative Algebra and Representation Theory.}



\author[F. Gawron, O. Genc]{Filip Gawron, Ozhan Genc}

\begin{document}

\begin{abstract}
We propose the definition of $\ell$-away ACM bundle on a polarized variety $(X, \cO_{X}(h))$. Then we give constructions of $\ell$-away ACM bundles on $(\p2 , \cO_{\p2}(1))$, $(\p1 \times \p1, \cO_{\p1 \times \p1}(1,1))$ and the anticanonically polarized blow up of $\p2$ up to three non collinear points. Also, we give the complete classification of $\ell$-away ACM bundles $\cE$ of rank 2 for values $1 \leq \ell \leq 2$ on $(\p2 , \cO_{\p2}(1))$. Similarly, on $(\p1 \times \p1, \cO_{\p1 \times \p1}(1,1))$, we give such a classification if $\mathrm{det}(\cE) = \cO_{\p1 \times \p1}(a,a)$ for some $a \in \bZ$. Moreover, we prove that the corresponding graded module $\mathrm{H}_*^1 ( \cE) = \underset{{t \in \bZ }}{\bigoplus}  \H1 (\cE (th))$ is connected, extending the similar result for bundles on $\p2$. 
\end{abstract}

\maketitle

\section{Introduction}
A vector bundle $\mathcal{E}$ on a smooth polarized variety $(X,\mathcal{O}_X(h))$ is called \textit{arithmetically Cohen-Macaulay (ACM)} if its cohomology satisfies $\mathrm{H}^i(X,\mathcal{E}(th))=0$ for all $t \in \mathbb{Z}$ and $0 < i < \mathrm{dim}(X)$. 
A well-known result of Horrocks states that a vector bundle $\mathcal{E}$ on $\mathbb{P}^n$ splits into line bundles if and only if $\cE$ is ACM. Although a decomposition result like that of Horrocks is not known for vector bundles on arbitrary polarized varieties, ACM bundles on such varieties are regarded as simpler in comparison. Consequently, ACM bundles on polarized varieties have received considerable attention among algebraic geometers, leading to a significant body of literature dedicated to their study..

The natural next step is to explore the properties of non-ACM bundles. 
\begin{definition} \label{defn:l-away}
	A non-ACM vector bundle $\cE$ on a polarized variety $(X,\cO_{X}(h))$ is called \textit{$\ell$-away ACM} if there are exactly $ \ell $ pairs $ (i, t) \in \bZ^2$ with $0<i<\mathrm{dim(X)}$ such that $\mathrm{h}^i(\cE(th))\neq0 $.
\end{definition}
This article aims to investigate the properties of such bundles for any positive integer $\ell$. The main problems are their existence for each $\ell$ and the classification of them for a given $\ell$. Also, the configuration of non zero intermediate cohomology groups is interesting to explore. 

The existence of these bundles is significant for existing literature, particularly in the context of weakly Ulrich and supernatural bundles (see Definition \ref{defn:weaklyUlrich} and Definition \ref{def:supernaturalbundle}). These types of bundles can have non-vanishing intermediate cohomology groups, making them relevant to the study of $\ell$-away ACM bundles. The article presents several results like corollaries \ref{corollary:supernaturalp2}, \ref{cor:higherrankweaklyUlrichP2}, \ref{remark:supernaturalquadric} and \ref{cor:higherrankweaklyUlrichP1XP1} in this direction.

In this work, we show the existence of $\ell$-away ACM bundles and give their classification on the Fano surfaces $(\p2, \cO_{\p2}(1))$, $(\p1 \times \p1, \cO_{\p1 \times \p1}(1,1))$ and the anticanonically polarized blow up of $\p2$ up to three non collinear points. 
\begin{definition}
	A vector bundle $\cE$ on $(X,\cO_{X}(h))$  is called \textit{special} if $\det (\cE) = ch$ for some integer $c$.
\end{definition}
In Section \ref{sec:Preliminaries}, we determine the common properties of special $\ell$-away ACM bundles of rank 2 on Fano surfaces. Also, we give some basic definitions and results concerning the theory of bounded derived category of coherent sheaves on a smooth variety. We finish the section with some properties of quivers and their representations.

In Section \ref{sec:l-awayACMBundlesonP2}, we give the classification of $1$-away and $2$-away ACM bundles of rank 2 on $\p2$. Then, we point out that they are weakly Ulrich. Also, we show the existence of $\mu$-stable $\ell$-away ACM bundle of rank 2 on $\p2$ for any $\ell > 1$ and that this bundle corresponds to a smooth point in the moduli space of $\mu$-stable sheaves lying in a single
component of dimension $\ell^2+2\ell-3$. We finish this section by showing the existence of simple $\ell$-away ACM bundle of any even rank by using representations of quivers. As a corollary, we are able to show the existence of weakly Ulrich bundles of any even rank on $\p2$. 

In Section \ref{sec:l-awayACMBundlesonP1XP1}, we analyse $\ell$-away ACM bundles on $\p1 \times \p1$. We give the complete classification of $\ell$-away ACM line bundles for any $\ell$ and the complete classification of special $1$-away and 2-away ACM bundles of rank 2. Moreover, we show that the module $\mathrm{H}_*^1 (\cE)$ is connected for any special rank 2 bundle, where the connectedness can be defined as follows:
\begin{definition}
	Let $\cE$ be vector bundle on a polarized variety $(X,\cO_X(h))$. We say that the module $\H1_* (\cE)  
	$ is connected if $\h1 (\cE(kh)) \neq 0$ and $\h1 (\cE(mh)) \neq 0$ with $ k < m$ implies $\h1 (\cE(th)) \neq 0$ for $ k < t <m$.
\end{definition}
The connectedness result mentioned above extends to $\p1 \times \p1$ where the similar result was proved in \cite[Proposition 3.1]{Bu92} for $\p2$. Additionally, for any $\ell$, we show the existence of simple, special $\ell$-away ACM bundles of rank 2, representing smooth points of the moduli space of simple sheaves lying in a single component of dimension $2\ell^2+2\ell-3$. Moreover, we show the existence of special $\ell$-away ACM bundle of any even rank. We cap off this section with the existence of weakly Ulrich bundles of any even rank on $\p1 \times \p1$. 

Lastly, in Section \ref{sec:blowup}, we give the complete classification of $\ell$-away ACM line bundles on the blow up of $\p2$ at a single point. Then using these line bundles, we show the existence of $\ell$-away ACM bundles of rank 2. We finish the section by showing the existence of $\ell$-away ACM bundles of rank 1 and rank 2 on the blow ups of $\p2$ at two distinct points or three non collinear points. Unfortunately, for these Fano surfaces, we have no complete classification like the case $\p2$ or $\p1 \times \p1$. The problem of connectedness of $\mathrm{H}_*^1 (\cE)$ for a special rank 2 bundle is open even on the blow up of $\p2$ at a single point.
\subsection{Notations and Conventions}
We always work over the field $\mathbb{C}$ of complex numbers. 

If $\cE$ is a vector bundle on a variety $X$, we denote by $\mathrm{h}^i (X, \cE)$, $\mathrm{hom} (X, \cE)$ and $\mathrm{ext}^i (X, \cE)$ the dimension of the complex vector spaces $\mathrm{H}^i (X, \cE)$, $\mathrm{Hom} (X, \cE)$ and $\mathrm{Ext}^i (X, \cE)$, respectively.
\section{Preliminaries}\label{sec:Preliminaries}
 
Note that, for an $\ell$-away ACM bundle $\cE$ on a polarized variety $(X,\cO_{X}(h))$, $\cE(t_0 h)$ is again an $\ell$-away ACM bundle for each integer $t_0$. 
 
\begin{definition} \label{defn:initialized}
A vector bundle $\cE$ on a polarized variety $(X,\cO_{X}(h))$ is called \textit{initialized} if $\mathrm{h}^0(\cE) \neq 0$ but $\mathrm{h}^0 (\cE(-h))=0$.
\end{definition}

So, when we classify $\ell$-away ACM bundles on a polarized variety, we can restrict to initialized ones.

\begin{definition} \label{defn:Fano}
A smooth projective variety $X$ is called a Fano variety if its anti-canonical divisor $-K_X$ is ample. The \textit{index} $i_X$ of a Fano variety is the greatest integer such that $K_X \cong \cO_X(-i_X h)$ for some ample line bundle $\cO_X(h) \in \mathrm{Pic}(X)$ and the ample divisor $h$ is called the fundamental divisor.
\end{definition}
Throughout the paper, when we call a variety Fano surface, we mean a smooth Fano surface polarized with the fundamental divisor. In this case, we have three cases with respect to $i_X$:

If $i_X = 3$, then our Fano surface is $(\p2 , \cO_{\p2}(1)).$

If $i_X = 2$, then  our Fano surface is $(\p1 \times \p1,  \cO_{\p1 \times \p1}(1,1) )$. More explicitly, let $\pi_i: \p1 \times \p1 \to \p1$ be the projection map onto the $i$th factor for $i=1,2$. Write $h_i$ for the pull-back of a point on $\p1$ via the map $\pi_i$. Then $\mathrm{A}(\p1 \times \p1) \cong \bZ[h_1,h_2]/(h_1^2,h_2^2)$. The canonical bundle $K_{\p1 \times \p1}$ is $-2h_1-2h_2$ where $h=h_1 +h_2$ is the fundamental divisor.

If $i_X = 1$, then  our Fano surface is isomorphic to a blow-up of $\p2$ at a set of points $\{p_1, p_2, \ldots ,p_r\}$ with $r \leq 8$ where no three of them lie on a line and no six of them lie on a conic. The fundamental divisor is $h = - K_X$.

Notice that, by the Serre duality, we have 
\begin{equation} \label{eqn:VanishingH0H2}
\h2(\cE(th))=0 \; \text{for} \; t \geq 1-c -i_X
\end{equation}
for a rank 2 special, initialized bundle on a Fano surface $X$.

\begin{remark}
Note that $\cE^{\oplus m}$ is an $\ell$-away ACM bundle of rank $rm$ if $\cE$ is $\ell$-away ACM bundle of rank $r$ on a polarized variety $ (X,h) $.
\end{remark}

When $(X,h)$ is a surface, we define the numbers
$$k_0 := \underset{t \in \bZ}{\operatorname{min}} \{ \h1 (\cE(th)) \neq 0 \} , \  \ s := \underset{t \in \bZ}{\operatorname{max}} \{ \h1 (\cE(th)) \neq 0 \}-k_0 .
$$
Notice that we always have $s \geq l-1$.

\begin{proposition}\label{prop:Eisextension}
	Let $\cE$ be a special initialized vector bundle of rank 2 on a Fano surface $X$ with $\det (\cE) = ch$. If there exists an integer $t_0$ such that $-1 - c \geq t_0 \geq 1- i_X$ and $\h1 (\cE (t_0 h))=0$ then $\cE$ can be obtained as an extension of two line bundles
		\begin{equation*}
		0 \longrightarrow \cO_{X}(D) \longrightarrow \cE  \longrightarrow \ \cO_X (ch-D) \longrightarrow 0
	\end{equation*}
where $D$ is effective or $D = 0$.
\end{proposition}

\begin{proof}
	Let $\overline{s} \in \H0 (\cE)$ and let its zero locus be $D \cup Z$ where $D$ is the union of the components of pure dimension 1 and Z a scheme of dimension 0. Then, we have
	\begin{equation}\label{ses:zerolocusquadric1}
		0 \longrightarrow \cO_{X}(D) \longrightarrow \cE  \longrightarrow \ \cI_Z (ch-D) \longrightarrow 0.
	\end{equation}
	where $\h0 (\cO_{X}(D)) \neq 0$ (since $D$ is effective or $D = 0$) and $\h0 (\cO_{X}(D-h)) = 0$ (since $\h0 (\cE(-h)) = 0$). Then, by tensoring \ref{ses:zerolocusquadric1} with $\cO_{X}(t_0 h)$ and considering its cohomology, we have
	$$
	\h1(\cI_Z((t_0 +c)h-D)) \leq \h1(\cE(t_0 h))+\h2(\cO_{X}(t_0 h + D)).
	$$
	By assumption $\h1(\cE(t_0 h)) = 0$. Since we assumed $t_0 \geq 1 -i_X $, $\h2(\cO_{X}(t_0 h + D))= \h0(\cO_{X}((-t_0 - i_X )h - D))=0$ by noticing that $D$ is effective or $D =0$. Therefore, $\h1(\cI_Z((t_0 +c) h-D))=0$. On the other hand, by tensoring the short exact sequence  
	\begin{equation}
		0 \longrightarrow \cI_Z \longrightarrow \cO_{X} \longrightarrow \cO_Z \longrightarrow 0
	\end{equation}
	with $\cO_{X}((t_0 +c) h-D)$ and considering its cohomology, we have
	$$
	\h0 (\cO_Z) \leq \h1(\cI_Z((t_0 + c)h-D)) + \h0 (\cO_{X}((t_0 +c)h - D)).
	$$
	Since  $t_0 +c \le -1$ by assumption, we have $\h0 (\cO_{X}((t_0 +c)h - D)=0$. Also, we showed that $\h1(\cI_Z((t_0 +c)h-D))=0$. Therefore, $\h0 (\cO_Z)=0$. So $Z$ is empty.
\end{proof}

\begin{theorem} \label{thm:GeneralBoundofk}
Let $\cE$ be a special $\ell$-away ACM bundle of rank 2 on a Fano surface $X$. Then the following items hold.
\begin{enumerate}[label=(\roman*)]
\item $c = -2k_0 -i_X -s$.
\item If $\cE$ is initialized then $k_0 \geq -s-2$.
\item If $\cE$ is initialized, indecomposable and  $i_X > 1$ then $k_0 \leq 1 -i_X$.
\end{enumerate}

\end{theorem}

\begin{proof}
\underline{(i):} First, note that
\[\h1 \big(\cE((k_0 -t) h) \big)= \h1 \big(\cEv((-k_0 +t -i_X)h) \big)=\h1 \big(\cE((-c -k_0 +t -i_X)h) \big)
\]
for each $t \geq 0$, hence $-c -k_0-i_X = k_0 + s$ thanks to the definitions of $k_0$ and $k_0 +s$. Therefore, $c = -2k_0 -i_X -s$.\\
\underline{(ii):} Secondly, note that $\mathrm{h}^0 \big(\cE(-h) \big)=0$ since $\cE$ is initialized. Also, by (\ref{eqn:VanishingH0H2}), we have
\[
\mathrm{h}^2 \big(\cE^{\vee}((1-i_X)h) \big)=\mathrm{h}^2 \big(\cE((-c +1-i_X)h) \big)=0. 
\]

If $\mathrm{h}^1 \big(\cE((-c +2-i_X)h) \big)=0$, then $\cE$ is $(-c+3 - i_X)$-regular. So $\mathrm{h}^1 \big(\cE(th) \big)=0$ for all $t \geq -c +2-i_X$. Therefore, $-c +2-i_X > k_0 +s$. Then, by combining this with item $(i)$, we have $k_0 \geq -1$.

If $\mathrm{h}^1 \big(\cE((-c +2-i_X)h) \big) \neq 0$, then $k_0 \leq -c+2 -i_X \leq k_0 +s$. Then, by item $(i)$, $ -s-2 \leq k_0 \leq 2$. Finally, we have $k_0 \geq -s-2$ since $ -1 \geq -s-2$.\\
\underline{(iii):} Assume on the contrary, $k_0 \geq -i_X +2$.

Notice that $\h1 (\cE((k_0 -1)h)) = 0$ by definition. Also, $k_0 -1 \geq 1- i_X$ by the assumption and $-1-c \geq k_0 -1$ by part (i). Then, by Proposition \ref{prop:Eisextension}, $\cE$ is an extension of two line bundles
	\begin{equation} \label{ses:extensionfrombound}
	0 \longrightarrow \cO_{X}(D) \longrightarrow \cE  \longrightarrow \ \cO_X (ch-D) \longrightarrow 0.
\end{equation}

This is a contradiction for the case $i_X = 3$, because there is no indecomposable rank 2 vector bundle obtained by extension of two line bundles.

When $i_X = 2$, let $D = ah_1 + bh_2$ with $a, b \geq 0$. The inclusion
$\cO_X(D-h) \subseteq \cE (-h)$ implies $ab = 0$, hence we can assume $D = ah_1$ for some $a \geq 0$ without loss of generality. We have $k_0 \geq i_X - 2 \geq 0$ by hypothesis, then item (i) yields $c \leq - 2$. Künneth’s theorem implies
\begin{eqnarray*}
	&&\mathrm{ext}^1(\cO_X (ch - D), \cO_X(D)) = \h1 (\cO_X (2D - ch)) \\
	&& = \h0(\cO_{\p1} (2a - c)) \h1 (\cO_{\p1} (-c)) + \h1 (\cO_{\p1} (2a - c)) \h0 (\cO_{\p1}(-c)).
\end{eqnarray*}
Thus, $a \geq 0$ and $-c \geq 2$ implies that sequence (\ref{ses:extensionfrombound}) splits, a contradiction.

Therefore, the assumption is false and $k_0 \leq -i_X +1$ for $i_X >1$.
\end{proof}

\begin{remark}
	Notice that, in the item (iii) of Theorem \ref{thm:GeneralBoundofk}, we have obtained the result for $i_X > 1$ by showing that there is no non trivial extension in the form \ref{ses:extensionfrombound}. However, this is not possible in the case of $i_X = 1$, because, in principle, we may have non nef divisor $D$ with very negative intersection number with an exceptional curve. 
\end{remark}
\begin{lemma}\label{lem:simpleext2vanish}
	Let $X$ be a Fano surface. If $\cE$ is a simple vector bundle on $X$, then $\Ext^2(\cE,\cE)=0$.
\end{lemma}
\begin{proof}
	One can prove the statement by mimicking the proof of \cite[Lemma 2.2]{ACG21}.
\end{proof}
Now, in this part of the paper, we will mention some details about the bounded derived category of coherent sheaves on a smooth varieties. We refer the reader to the paper \cite{G--K} for more details. Let $D^b (X)$ denotes the bounded derived category of coherent sheaves on a smooth variety $X$. For each sheaf $E$ and integer $k$ we denote by $E[k]$ the trivial complex defined as
$$
E[k]^i=\left\lbrace\begin{array}{ll} 
0\quad&\text{if $i\ne k$,}\\
E\quad&\text{if $i=k$.}
\end{array}\right.
$$
with the trivial differentials: we will often omit $[0]$ in the notation. A coherent sheaf $E$ is {\sl exceptional} as trivial complex in $D^b(X)$ if
$$
\Ext_X^k\big(E, E \big) = \left\lbrace\begin{array}{ll} 
0\quad&\text{if $k\ne 0$,}\\
\bC\quad&\text{if $k=0$.}
\end{array}\right.
$$
 An ordered set of exceptional objects $(E_0,\dots, E_s)$ is an {\sl exceptional collection} if $\Ext_{X}^k\big(E_i, E_j\big) = 0$ when $i > j$ and $k\in\bZ$. An exceptional collection is {\sl full} if it generates $D^b(X)$ and {\sl strong} if $
\Ext_{X}^k\big(E_i, E_j\big) = 0$ when $i < j$ and $k\in\bZ\setminus\{\ 0\ \}$.

If $(E_0[k_0],\dots,E_s[k_s])$ is full, then there exists a unique collection $(F_0,\dots,F_s )$ satisfying 
$$
\Ext_{X}^{k-k_i}\big(E_i,F_j\big)=\left\lbrace\begin{array}{ll} 
\bC\quad&\text{if $i+j= s$ and $i=k$,}\\
0\quad&\text{otherwise,}
\end{array}\right.
$$
Such a collection is called the {\sl right dual collection} of $(E_0[k_0],\dots,E_s[k_s])$ (see \cite[Proposition 2.6.1]{G--K}). 

Finally, let us review the definition and basic properties of quivers. For an introduction to the topic, see \cite{DW2017}.

\begin{definition}
	A quiver $Q$ is a 4-tuple $(Q_0, Q_1, s, t)$ where $Q_0$ and $Q_1$ are sets, whose elements are called \emph{vertices} and \emph{arrows} respectively, and two maps $s,t: Q_1 \to Q_0$ that sends each arrow to its \emph{source} and \emph{target} respectively. A quiver $Q$ is called \emph{finite} if $Q_0$ and $Q_1$ are finite sets. It is called \emph{acyclic} if it does not contain any oriented cycle.
\end{definition}

\begin{definition}\label{defn:representquiver}
	A \emph{representation} $V$ of the quiver $Q$ is a set of finite-dimensional vector spaces $V_i$ for every $i \in Q_0$ and linear maps $V_a: V_i \to V_j$ for every arrow $a: i \to j$. $V$ is called \emph{trivial} if all $V_i = 0$; in this case, we just write $V=0$. Given two representations $V$ and $W$, a \emph{morphism} $\phi: V \to W$ is a collection of linear maps $\phi_i : V_i \to W_i$ such that for every arrow $a: i \to j$, the diagram
	\[
	\begin{tikzcd}
		V_i \arrow[r, "V_a"] \arrow["\phi_i", d] & V_j \arrow[d, "\phi_j"] \\
		W_i \arrow[r, "W_a"] & W_j
	\end{tikzcd}
	\]
	commutes.
\end{definition}
We can define a category $\mbox{Rep}(Q)$ where the objects are representations of $Q$ and
the morphisms are as defined in Definition \ref{defn:representquiver}. The direct sum of quiver representations, monomorphisms, epimorphisms, isomorphisms, and kernels and cokernels can be defined in a straightforward manner (see \cite[Chapter 1.2]{DW2017}).

Let $V$ be a representation of a quiver $Q$. We can introduce the extension group $\mbox{Ext}^1(V, -)$ as the derived functor of $\mbox{Hom} (V, - )$.

\begin{lemma}
	For any two quiver representations $V$ and $W$, $\mbox{Ext}^{i}(V,W)=0$ for $i \geq 2$.
\end{lemma}
\begin{proof}
See \cite[Lemma 2.4.3]{DW2017}.
\end{proof}
This allows us to define the \emph{Euler characteristic} as follows.
\begin{definition}
	Given two quiver representations $V$ and $W$, the \emph{Euler characteristic} is defined by
	\[
	\chi(V, W) = \mbox{hom}(V,W) - \mbox{ext}(V,W)
	\]
where $\mbox{hom}(V,W)$ and $\mbox{ext}(V,W)$ stand for dimensions of $\mbox{Hom}(V,W)$ and $\mbox{Ext}^1 (V,W)$.
\end{definition}
\begin{definition}
	If the quiver $Q$ is finite and the $V_i$ were assumed to be finite-dimensional, the dimensions of $V_i$ can be arranged in a tuple of integers $\underline{d} = \underline{d}(V) \in \mathbb{Z}^{Q_0}$, called \emph{dimension vector}.
\end{definition}
\begin{definition}
	Given two dimension vectors $\alpha, \beta \in \mathbb{Z}^{Q_0}$, the \emph{Euler form} is defined by
	\[
	\langle \alpha, \beta \rangle = \sum_{i \in Q_0} \alpha_i \beta_i - \sum_{a \in Q_1} \alpha_{s(a)} \beta_{t(a)}.
	\]
\end{definition}

The Euler characteristic and the Euler form are related by the following result.
\begin{proposition}\label{prop:eulerdimensionvector}
	Given two quiver representations $V$ and $W$, we have
	\[
	\chi(V,W) = \langle \underline{d}(V), \underline{d}(W) \rangle.
	\]
\end{proposition}
\begin{proof}
	See \cite[Proposition 2.5.2]{DW2017}.
\end{proof}
Let us finish this section with the following:
\begin{theorem}\label{thm:bondal}
Suppose there exists a full strong exceptional collection $(E_0, \cdots, E_s)$ on $D^b (X)$ for a smooth variety $X$. Then there exists a quiver with relations $(Q; J)$, and a triangulated equivalence
\[
\Phi : D^b(X) \longrightarrow D(Q; J).
\]
\end{theorem}
\begin{proof}
	See \cite[Theorem 2.6.3]{M17}.
\end{proof}

\section{$\ell$-away ACM Bundles on $\p2$}\label{sec:l-awayACMBundlesonP2}

In this section, we give the complete classification of initialized 1-away and 2-away ACM bundles of rank 2 on Fano surface $\p2$. Also, we show the existence of simple initialized $\ell$-away ACM bundles of any even rank on $\p2$ for any $ \ell >1$. Notice that the fundamental divisor $h$ of $\p2$ is determined by the line bundle $\cO_{\p2}(1)$.

Since, $\operatorname{Pic} (\p2) \cong \bZ<h>$, in this section we will use the notation $\cE(t)$ for $\cE(th)$ for an integer $t$. 

The Riemann-Roch theorem for a rank 2 vector bundle $\cE$ on $(\p2,\cO_{\p2}(1))$ with $\mathrm{det}(\cE) = c$ for some integer $c$ gives
\begin{equation} \label{eqn:RRoch}
	\chi \big( \cE(t) \big) = \frac{1}{2}\big( 2 t^2 + (2c + 6)t + c^2 + 3 c  -2c_2 (\cE) +4\big).
\end{equation}

First, let us state the standard Beilinson theorem for $\p2$. This proposition will be useful for the classification results.
\begin{proposition}\label{prop:monadonP2}
	Let $\cE$ be a vector bundle on $\p2$. Then $\cE$  is the cohomology in degree $0$ of a complex ${\cC}^\bullet$ with $r^{th}$--module
	$$
	{\cC}^r:=\bigoplus_{i+j=r}\Ext_{\p2}^{j}\big(\cO_{\p2}(-i),\cE\big)\otimes \Omega^{-i}_{\p2}(-i).
	$$
\end{proposition}
\begin{proof}
	See \cite[Theorem 3.1.3]{OSG80}.
\end{proof}
Now, we are ready to give the complete classifications of $1$-away and $2$-away ACM bundles of rank 2 on $\p2$.

\begin{theorem}\label{thm:1awayclassificationP2}
$\cE$ is an initialized $1$-away ACM bundle of rank 2 on $\p2$ if and only if $\cE \cong \Omega_{\p2}(2)$.
\end{theorem}

\begin{proof}
First, assume that, $\cE$ is initialized $1$-away ACM of rank 2. Then, by Theorem \ref{thm:GeneralBoundofk}, $k_0 \geq -2$. Since  $\cE$ is $1$-away ACM, $k_0 = -2$. So $c=1$; and so $\rh^1 (\cE(t))=0$ for all $t \in \bZ \setminus \{-2\}$. Also, note that $\rh^2 (\cE(t))=0$ for $t \geq -3$ by (\ref{eqn:VanishingH0H2}). Lastly, $\rh^0 (\cE(t))=0$ for $ t \leq -1$ because $\cE$ is initialized. So, by Equation (\ref{eqn:RRoch}), 
\[0=\chi(X,\cE(-1))= 1- c_2(\cE) \Longrightarrow c_2(\cE)=1. 
\]
Then, again by Equation (\ref{eqn:RRoch}), we have $-1=\chi(\cE(-2))=-\rh^1(\cE(-2))$. 
So, we deduce that $\cE \cong \Omega_{\p2}(2)$ by applying Proposition \ref{prop:monadonP2} to $\cE(-1)$.

The converse is an easy application of Bott's Theorem.
\end{proof}

\begin{theorem}\label{thm:2awayclassificationP2}
Let $\cE$ be an initialized rank 2 vector bundle on $\p2$. Then, $\cE$ is a $2$-away ACM if and only if $\cE$ fits into one of the following sequences:

\begin{enumerate}[label=(\roman*)]
\item $0 \longrightarrow \cE \longrightarrow \Omega^{\oplus 2}_{\p2}(3)  \longrightarrow \cO^{\oplus 2}_{\p2}(2) \longrightarrow 0 $
\item $0 \longrightarrow \cE \longrightarrow \Omega_{\p2}(2) \oplus \cO_{\p2} \longrightarrow \cO_{\p2}(1) \longrightarrow 0$.
\end{enumerate}

\end{theorem}

\begin{proof}
Let $\cE$ be a $2$-away ACM bundle of rank 2 on $\p2$. Then $s=1$ by \cite[Proposition 3.1]{Bu92}. So, 
by Theorem \ref{thm:GeneralBoundofk}, $-3 \leq k_0 \leq -2$ and $c=-2k_0 -4$.

If $k_0=-3$, then $\h1 \big( \cE (-1) \big)=0$. Also, by (\ref{eqn:VanishingH0H2}), $\mathrm{h}^{2} \big( \cE (-1) \big)=0$. Since $\cE$ is initialized, $\mathrm{h}^{0} \big( \cE (-1) \big)=0$. So, by Equation (\ref{eqn:RRoch}), $0 = \chi \big(\cE (-1) \big)$ gives rise to $c_2 (\cE)=3$. Similarly, $\mathrm{h}^{i} \big( \cE (t) \big)=0$ for $t= -4, -3, -2$ and $i= 0,2$. Then, by Equation (\ref{eqn:RRoch}), we have the following Table \ref{table2}

\begin{table}[H] 
\centering
\bgroup
\def\arraystretch{1.5}
\begin{tabular}{cccccc}
\cline{2-4}
\multicolumn{1}{c}{} &\multicolumn{1}{|c|}{0} & \multicolumn{1}{c|}{0} & \multicolumn{1}{c|}{0} & $q=2$ \\ 
\cline{2-4}
\multicolumn{1}{c}{} &\multicolumn{1}{|c|}{0} & \multicolumn{1}{c|}{2} & \multicolumn{1}{c|}{2}  & $q=1$ \\ 
\cline{2-4}
\multicolumn{1}{c}{} &\multicolumn{1}{|c|}{0} & \multicolumn{1}{c|}{0} & \multicolumn{1}{c|}{0}  & $q=0$ \\ 
\cline{2-4}
&$p=-4$ & $p=-3$ & $p=-2$ & 
\end{tabular}
\egroup
\caption{The values of $h^{q}\big(\p2,\cE(ph)\big)$ for $\p2$}
\label{table2}
\end{table}

If we apply Proposition \ref{prop:monadonP2} to $\cE(-2)$ then $\cE$ can be represented as
$$
0 \longrightarrow \cE \longrightarrow \Omega^{\oplus 2}_{\p2}(3)  \longrightarrow \cO^{\oplus 2}_{\p2}(2) \longrightarrow 0
$$
which is the item (i).

If $k_0=-2$, then by (\ref{eqn:VanishingH0H2}), $\mathrm{h}^{2} \big(\cE(t) \big)=0$ for $t=-1,-2$. Also, $\mathrm{h}^{0} \big(\cE(t) \big)=0$ for $t=-1,-2,-3$ since $\cE$ is initialized. So, by Equation (\ref{eqn:RRoch}), $\mathrm{h}^1 \big(\cE (t) \big) = - \chi \big(\cE (t) \big)= c_2(\cE)$ for $t=-1,-2$. Also, since $\mathrm{h^{i} \big(\cE(-3) \big)}=0$ for $i=0,1$, $\mathrm{h}^2 \big(\cE (-3) \big) = \chi \big(\cE (-3) \big)= 2-c_2(\cE)$. Note that $\h1 (\cE(-2))$ is positive since $k_0=-2$. Also, $\h2(\cE(-3)) = \h0(\cE)$ is  positive by definition. So $c_2(\cE)$ and $2-c_2(\cE)$ are positive; that is, $c_2(\cE)=1$. Then we have the following Table \ref{table3}

\begin{table}[H]
\centering
\bgroup
\def\arraystretch{1.5}
\begin{tabular}{cccccc}
\cline{2-4}
\multicolumn{1}{c}{} &\multicolumn{1}{|c|}{1} & \multicolumn{1}{c|}{0} & \multicolumn{1}{c|}{0} & $q=2$ \\ 
\cline{2-4}
\multicolumn{1}{c}{} &\multicolumn{1}{|c|}{0} & \multicolumn{1}{c|}{1} & \multicolumn{1}{c|}{1}  & $q=1$ \\ 
\cline{2-4}
\multicolumn{1}{c}{} &\multicolumn{1}{|c|}{0} & \multicolumn{1}{c|}{0} & \multicolumn{1}{c|}{0}  & $q=0$ \\ 
\cline{2-4}
&$p=-3$ & $p=-2$ & $p=-1$ & 
\end{tabular}
\egroup
\caption{The values of $h^{q}\big(\p2,\cE(ph)\big)$ for $\p2$}
\label{table3}
\end{table}

Similarly, if we apply Proposition \ref{prop:monadonP2} to $\cE(-1)$ then $\cE$ can be represented as
$$
0 \longrightarrow \cE \longrightarrow \Omega_{\p2}(2) \oplus \cO_{\p2} \longrightarrow \cO_{\p2}(1) \longrightarrow 0
$$
which is the item (ii).

Conversely, by applying Bott's Theorem, one can see that kernel bundles in items (i) and (ii) are $2$-away ACM bundles of rank 2 on $\p2$.
\end{proof}
In the theory of vector bundles, there are particular bundles which are called weakly Ulrich. 
\begin{definition}\label{defn:weaklyUlrich}
	Let $\cE$ be a vector bundle on a smooth, polarized variety $(X,\cO_X(h))$ of dimension $n$. Then, $\cE$ is weakly Ulrich if 
	\begin{itemize}
		\item $\mathrm{h}^i (\cE(-mh))=0$ for $1 \leq i $ and $m \leq i-1$
		\item $\mathrm{h }^i (\cE(-mh))=0$ for $i \leq n-1$ and $ i+2 \leq m$.
	\end{itemize}
\end{definition} 
For the interested reader, the main references are \cite{CRL21} and \cite{ESW03}. 

Now, we have the complete list of indecomposable weakly Ulrich bundles of rank 2 on $\p2$.
\begin{corollary}
	All indecomposable weakly Ulrich bundles of rank 2 on $\p2$ are $1$-away and $2$-away ACM bundles on $\p2$ up to a shift.
\end{corollary}
\begin{proof}
	Notice that a non-ACM weakly Ulrich bundle on a surface is $\ell$-away ACM for $\ell \leq 2$. Using the classification of rank two $\ell$-away ACM bundles on $\p2$ with $\ell \leq 2$ in Theorems  \ref{thm:1awayclassificationP2} and \ref{thm:2awayclassificationP2}, it is easy to check that they are weakly Ulrich up to a shift.
\end{proof}
When $\ell \geq3$, it gets complicated to give the complete classification of $\ell$-away ACM bundles of rank 2 on $\p2$. In the meantime, there always exists $\mu$-stable $\ell$-away ACM bundle of rank 2 on $\p2$ for all $\ell$.
\begin{theorem}\label{thm:existenceofgenerallawayP2}
The bundle $\cE$, fitting in the displayed sequence,
\begin{eqnarray}\label{ses:existenceofgenerallaway}
0 \longrightarrow \cE \longrightarrow \cO^{\oplus \ell +2}_{\p2}(\ell)  \longrightarrow \cO^{\oplus \ell}_{\p2}(\ell + 1) \longrightarrow 0
\end{eqnarray}
is a $\mu$-stable $\ell$-away ACM bundle of rank 2 on $\p2$. Also, the point corresponding to $\cE$ in the moduli space of $\mu$-stable sheaves on $\p2$ is smooth and lies in a component of dimension $\ell^2 +2\ell -3$. 
\end{theorem}

\begin{proof}
Since the general morphism $\cO^{\oplus \ell +2}_{\p2}(\ell)  \longrightarrow \cO^{\oplus \ell}_{\p2}(\ell + 1)$ is surjective (see \cite[Proposition 2.1]{Fr00}), there exists such a bundle $\cE$.\\ 
First note that $\det(\cE)= \cO_{\p2} \big(  \ell (\ell+2) -\ell(\ell + 1) \big) = \cO (\ell)$.

Now, let us show that $\cE$ is initialized. The cohomology of sequence \ref{ses:existenceofgenerallaway} returns
\[
\h0 \big( \cE \big) \geq (\ell +2) \h0 \big( \cO_{\p2}(\ell) \big) -  \ell \ \h0 \big( \cO_{\p2}(\ell +1) \big) = \ell+2.
\]
Now, let us consider $\h0 \big( \cE(-1) \big)=\h2 \big( \cEv (-2) \big)=\h2 \big( \cE(-\ell-2) \big)$. Then, similarly 
\begin{align*}
	\h2 \big( \cE(-\ell-2) \big) &\leq \ell \  \h1 \big( \cO_{\p2}(-1) \big) + (\ell +2) \h2 \big( \cO_{\p2}(-2) \big)=0.
\end{align*}
So, $\h0 \big( \cE(t) \big)=0$ for $t \leq -1$. Therefore, $\cE$ is initialized.

Now, for proving that $\cE$ is $\ell$-away ACM, it remains to show $\h1 \big( \cE(t) \big) \neq 0$ for  $-\ell-1 \leq t \leq -2$ and  $\h1 \big( \cE(t) \big) = 0$ otherwise. First of all, $\h1 \big( \cE(-\ell-1) \big) = \ell \ \h0 \big( \cO_{\p2} \big)=\ell$. If $-\ell \leq t \leq -2$, then
\begin{align*}
	\h1 \big( \cE(t) \big) &= \ell \  \h0 \big( \cO_{\p2}(t+\ell +1) \big) - (\ell+2) \h0 \big( \cO_{\p2}(t+ \ell) \big)\\
	&=(t+\ell+2)(-t-1) > 0. 
\end{align*}
Also, if $t \leq -\ell-2$ then 
\[
	\h1 \big( \cE(t) \big) \leq \ell \  \h0 \big( \cO_{\p2}(t+\ell +1) \big) + (\ell+2) \h1 \big( \cO_{\p2}(t+ \ell) \big) = 0. 
\]
Then, by the Serre duality, $\h1 \big( \cE(t) \big)=0$ for $t\geq-1$ too. Hence, $\cE$ is a $\ell$-away ACM bundle of rank 2. 

It is $\mu$-stable by Hoppe's Criterion (see \cite[Corollary 4]{JMPS17}), since it is initialized and it has positive first Chern class.

Finally, since it is $\mu$-stable, the  moduli space at the corresponding point is smooth of dimension $\h1 \big( \cEv \otimes \cE \big)$ if $\h2 \big( \cEv \otimes \cE \big)=0$ (see \cite[Corollary 4.5.2]{HL10}). Since we showed that $\cE$ is $\mu$-stable, $\cE$ is simple. Then, $\mathrm{ext^2}(\cE,\cE)=\h2(\cEv \otimes \cE )=0$ by Lemma \ref{lem:simpleext2vanish}. Also, $\mathrm{hom}(\cE,\cE)=\h0(\cEv \otimes \cE )=1$ by simplicity of $\cE$. Therefore, the dimension of moduli space can be computed by the equation
\begin{equation}\label{eqn:dimmodulip2}
	\h1(\cEv \otimes \cE )=1-\rchi(\cEv \otimes \cE ).
\end{equation}
 If we tensor the sequence \ref{ses:existenceofgenerallaway} by $\cEv \simeq \cE(-\ell)$ and use the additivity of the Euler characteristic under exact sequences, the equation \ref{eqn:dimmodulip2} will be 
\begin{equation}\label{eqn2:dimmodulip2}
	\h1(\cEv \otimes \cE )=1- (\ell+2) \cdot \rchi(\cE)  + \ell \cdot \rchi(\cE(1)).
\end{equation}
Then, by applying the equation \ref{eqn:RRoch} for $\cE$ and $\cE(1)$, we get $\h1(\cEv \otimes \cE )=\ell^2 +2\ell -3$.
\end{proof}
The bundle constructed in Theorem \ref{thm:existenceofgenerallawayP2} is supernatural and important in the Boij–S\"{o}derberg theory.
\begin{definition}\label{def:supernaturalbundle}
	A bundle $\cE$ on a  polarized variety $(X,\cO_X(h))$ of dimension $n$ is called \textit{supernatural} if for each $j \in \mathbb{Z}$ there is at most one $i$ such that $\mathrm{h}^i(\cE(jh)) \neq 0$ and the Hilbert polynomial of $\cE$ has $n$ distinct integral roots. 
\end{definition}
\begin{corollary}\label{corollary:supernaturalp2}
	The bundle $\cE$ constructed in Theorem \ref{thm:existenceofgenerallawayP2} is a supernatural bundle and its Hilbert polynomial has integral roots  $-\ell-2$ and $-1$.
\end{corollary}
\begin{proof}
	We have shown in the proof of Theorem \ref{thm:existenceofgenerallawayP2} that
	\begin{eqnarray*}
		\h0(\cE(t)) &=& 0 \text{ for } t \leq -1\\
		\h1(\cE(t)) &=& 0 \text{ for } t \leq - \ell -2 \text{ and } t \geq -1\\
		\h2(\cE(t)) &=& 0 \text{ for } t \geq - \ell -2.
	\end{eqnarray*}
\end{proof}
\begin{lemma}\label{lem:supernaturalopencondition}
	Let $\cE$ be a supernatural $\ell$-away ACM bundle on a polarized variety $(X,\cO_X(h))$ of dimension $n$ and represented in a flat family. Then there exists an open subset in this family whose points corresponds to supernatural $\ell$-away ACM bundles.
\end{lemma}
\begin{proof}
	Since the dimension of cohomology is an upper semi-continuous function on a flat family, then there exists an open subset $O$ in the family around the point corresponding to $\cE$ where for any bundle $\cF \in O$ we have $\mathrm{h}^i(\cF(jh)) \leq \mathrm{h}^i(\cE(jh))$ for each $i,j$. So, if $\mathrm{h}^i(\cE(jh))=0$ is then $\mathrm{h}^i(\cF(jh))=0$. Hence, the Hilbert polynomial of $\cF$ has the same integral roots as $\cE$. So, the Hilbert polynomial of $\cF$ has at least $n$ integral roots.
	
	If $\mathrm{h}^{i_0}(\cE(j_0 h)) \neq 0$ then $\mathrm{h}^{i_0}(\cF(j_0 h)) \neq 0$. If $\mathrm{h}^{i_0}(\cF(j_0h))$ were $0$ then $\cF$ would have one more integral root at $j_0$ which is a contradiction. Because, the degree of the Hilbert polynomial is $n$ and it may have at most $n$ integral root. Therefore, $\cF$ is supernatural, $\ell$-away ACM.
\end{proof}
\begin{theorem}\label{thm:higherrankp2}
 There exist simple $\ell$-away ACM bundles $\cF$ on $\p2$ of any even rank $r=2m$ and for any $\ell>1$. Also, the point corresponding to $\cF$ in the moduli space of simple sheaves on $\p2$ is smooth and lies in a component of dimension $m^2(\ell^2+2\ell-4)+1$.
\end{theorem}

\begin{proof}
Let $\cE$ be the bundle in the short exact sequence \ref{ses:existenceofgenerallaway}. Then by Theorem \ref{thm:bondal} (see also \cite[Section 5.1]{M17}), $\mathcal{E}$ is mapped to the complex of quiver representation
\[
R\mbox{Hom}(\mathcal{O}_{\p2}(2), \mathcal{E}) \triplerightarrow R\mbox{Hom}(\Omega_{\p2}(3), \mathcal{E}) \triplerightarrow R\mbox{Hom}(\mathcal{O}_{\p2}(1), \mathcal{E}).
\]
The right vertex disappears because $\Ext^i \big( \mathcal{O}_{\p2}(1), \mathcal{E} \big) = \mathrm{H}^{i} \big(\cE (-1) \big)=0$ for all $i$ by Corollary \ref{corollary:supernaturalp2}. Similarly, by applying Corollary \ref{corollary:supernaturalp2}, we have that the graded vector space at the left vertex is concentrated in degree 1. Also, the cohomology of sequence (\ref{ses:existenceofgenerallaway}) tensored with $\Omega (-\ell)$ returns $\mathrm{h}^i(\cE \otimes \Omega (-\ell)) = 0 $ for $i=0,2$ and $\mathrm{h}^1(\cE \otimes \Omega (-\ell)) = \ell +2$. Since $\Ext^i \big( \Omega_{\p2}(3), \mathcal{E}) \big) = \mathrm{H}^i \big( \cE \otimes T_{\p2} (-3) \big) = \mathrm{H}^{2-i} \big( \cE \otimes \Omega_{\p2} (-\ell) \big)$, the middle vertex is also concentrated in degree 1.

Therefore, we have the quiver with two vertices and three arrows (all have the same source and the same target)
\[
1 \triplerightarrow 2
\]
with the representation of dimension vector $\underline{d} = (d_1, d_2)$ where $d_1 = \h1 (\cE (-2)) = \ell$ and $d_2 = \h1 (\cE \otimes \Omega (-l)) = \ell +2$.\\
Then, by Proposition \ref{prop:eulerdimensionvector}, we have $\chi(\underline{d}, \underline{d})= d_1^2 + d_2^2 -3d_1d_2 = -\ell^2 -2\ell + 4 < 0$. By \cite[Proposition 1.6 (a)]{Kac1980}, $\underline{d}$ is an imaginary root. Then, by \cite[Lemma 2.1 (e) and Lemma 2.7]{Kac1980}, $\underline{d}$ is a Schur root; that is, a general representation of a dimension
vector $\underline{d}$ is indecomposable. By \cite[Theorem 3.7]{Schofield1992}, $m \underline{d}$ is a Schur root for every $m \geq 1$. Hence the general element in $Rep(m \underline{d})$ is indecomposable. 

Now, first notice that $\mathcal{E}^{\oplus m}$ is a supernatural $\ell$-away bundle, since $\cE$ is so. Secondly, $\mathcal{E}^{\oplus m}$ corresponds to a quiver representation with dimension vector $m \underline{d}$ via the map $\Phi$ in Theorem \ref{thm:bondal}. Also, since the general element in $Rep(m \underline{d})$ is indecomposable, $\Phi^{-1}$ can not map a generic element in $Rep(m \underline{d})$ to a decomposable sheaf. Therefore, there exists an indecomposable $\ell$-away ACM bundle $\cF$ of rank $2m$, since the condition of being supernatural $\ell$-away ACM is open by Lemma \ref{lem:supernaturalopencondition}.

Lastly, notice that $\mathrm{ext}^2 (\cF, \cF) =0$ by Lemma \ref{lem:simpleext2vanish} since $\cF$ is simple. So, the point corresponding to $\cF$ in the moduli space of simple sheaves on $\p2$ is smooth and its dimension can be computed by considering the Euler form as
\begin{equation*}
	1 - \rchi(m \underline{d}, m \underline{d}) = 1 - m^2 \cdot \rchi(\underline{d}, \underline{d})=m^2(\ell^2+2\ell-4)+1.
\end{equation*}
\end{proof}

\begin{corollary}\label{cor:higherrankweaklyUlrichP2}
	There exist indecomposable weakly Ulrich bundle of any even rank on $\p2$.
\end{corollary}
\begin{proof}
	Let $\cF$ be a 2-away ACM bundle of rank $r=2m$ constructed in Theorem \ref{thm:higherrankp2}. Then $\cF(-1)$ is weakly Ulrich.
\end{proof}

\section{$\ell$-away ACM Bundles on $\p1 \times \p1$}\label{sec:l-awayACMBundlesonP1XP1}
In this section, we will give similar results for Fano surface $\p1 \times \p1$. Recall that $h=h_1 +h_2$ is the fundamental divisor.

The Riemann-Roch theorem for a rank 2 vector bundle $\cE$ on $\p1 \times \p1$ with $c_i(\cE)=c_i$  gives
\begin{equation} \label{eqn:RRochforp1Xp1}
\rchi \big(\cE \otimes \cO(ah_1 + bh_2) \big) = \frac{1}{2}\big(c^2_1 + \big( (2a+2)h_1 +(2b+2)h_2  \big)c_1-2c_2+4(ab+a+b+1) \big).
\end{equation}

\begin{theorem}\label{thm:lawayLineBundlesonQuadric}
	$\mathcal{L}=\cO_{X}(ah_1 +bh_2)$ is an initialized $\ell$-away ACM line bundle on $X=\p1 \times \p1$ if and only if $\mathcal{L}=\cO_{X}((\ell+1)h_1)$ or $\mathcal{L}=\cO_{X}((\ell+1)h_2)$. In this case $k_0 = -\ell-1$ and $s=\ell-1$.
\end{theorem}

\begin{proof}
	Assume that $\mathcal{L}=\cO_{X}(ah_1 +bh_2)$ is an initialized $\ell$-away ACM line bundle on $X$. Since, by definition, $\h0 (\mathcal{L}) \neq 0$, $a,b \geq 0$. Also, since $\h 0 (\mathcal{L}(-h)) = 0$, $a-1 <0$ or $b-1<0$. Because of the symmetry, we can assume, without loss of generality, that $a-1<0$ . Then, $a=0$ and $b \geq 0$.
	
	Since $\H1 ( \mathcal{L}((k_0 -1)h))=0$ and $\H1 (\mathcal{L}((k_0)h)) \neq 0$, by Künneth’s theorem, we have
	\begin{eqnarray*}
		& \h0 (\cO_{\p1}(k_0 -1)) \h0 (\cO_{\p1}(-1-b-k_0))=0\\
		& \h0 (\cO_{\p1}(-k_0 -1)) \h0 (\cO_{\p1}(-1+b+k_0))=0\\
		& \h0 (\cO_{\p1}(k_0)) \h0 (\cO_{\p1}(-2-b-k_0))+ \h0 (\cO_{\p1}(-k_0 -2)) \h0 (\cO_{\p1}(b+k_0))\neq0.
	\end{eqnarray*}
	Then we have $k_0 \leq -2$ and $b=-k_0$. So, $\mathcal{L}=\cO_{X}(-k_0 h_2)$.
	
	Similarly, since $\h1 ( \mathcal{L}((k_0 +s +1)h))=0$ and $\h1 (\mathcal{L}((k_0 +s)h)) \neq 0$ , we have
	\begin{eqnarray*}
		& \h0 (\cO_{\p1}(k_0 +s+1)) \h0 (\cO_{\p1}(-s-3))=0\\
		& \h0 (\cO_{\p1}(-k_0 -s-3)) \h0 (\cO_{\p1}(s+1))=0\\
		& \h0 (\cO_{\p1}(k_0+s)) \h0 (\cO_{\p1}(-2-s))+ \h0 (\cO_{\p1}(-k_0 -s-2)) \h0 (\cO_{\p1}(s))\neq0.
	\end{eqnarray*}
	Then, using the fact that $s \geq \ell-1$, we have $k_0 +s = -2$.
	
	Lastly, it remains to prove $s=\ell-1$. Assume on the contrary that there exists an integer $t_0$ with $k_0 < t_0 < k_0 +s$ such that $\h1 (\mathcal{L}(t_0 h))=0$. Then, by Künneth’s theorem, both $\h0 (\cO_{\p1}(t_0)) \h0 (\cO_{\p1}(k_0-t_0-2))$ and $\h0 (\cO_{\p1}(-t_0 -2)) \h0 (\cO_{\p1}(t_0 -k_0))$ are 0. Due to our assumption, $t_0 -k_0 >0$, which forces $\h0 (\cO_{\p1}(-t_0 -2))=0$.  So, $-1 \leq t_0$, which implies $-1<k_0+s$ because of the inequality $t_0 < k_0 +s$. But, we showed that $k_0 +s = -2$; that is, our assumption that there exists an integer $t_0$ with $k_0 < t_0 < k_0 +s$ such that $\h1 (\mathcal{L}(t_0 h))=0$ is false. Therefore, $s=\ell-1$ and so $k_0=-\ell-1$. Thus $\mathcal{L}=\cO_{X}((\ell+1) h_2)$. If we start from $a \geq 0$ and $b = 0$ the same argument yields $\mathcal{L}=\cO_{X}((\ell+1) h_1)$.
	
	Conversely, one can show easily that $\mathcal{L}=\cO_{X}((\ell+1) h_1)$ and $\mathcal{L}=\cO_{X}((\ell+1) h_2)$ are $\ell$-away ACM line bundles on $X$ using Künneth’s theorem.
	\end{proof}
In this part of the paper, we will give the complete classification of special initialized 1-away and 2-away ACM bundles of rank 2 on $\p1 \times \p1$. For this purpose, we need the following proposition.

\begin{proposition}\label{prop:monadonP1XP1}
	Let $\cE$ be a vector bundle on $X=\p1 \times \p1$. Then $\cE$  is the cohomology in degree $0$ of a complex ${\cC}^\bullet$ with $r^{th}$--module
	$$
	{\cC}^r:=\bigoplus_{i+j=r}\Ext_{X}^{{j-v_i}}\big(E_{-i},\cE\big)\otimes F_{3+i}
	$$
	where $(E_0,E_1,E_2,E_3) = \big(\cO_{X}, \cO_{X}(h_1), \cO_{X}(h_2), \cO_{X}(h_1+h_2) \big)$, $(v_0, v_1, v_2, v_3) = (0,0,1,1)$ and $(F_0,F_1,F_2, F_3) = 	\big( \cO_{X}(-h_1 -h_2), \cO_{X}(-h_2), \cO_{X}(-h_1), \cO_{X} \big)$.
\end{proposition}
\begin{proof}
	It follows from \cite[Section 2.7.3]{G--K}.
\end{proof}

\begin{theorem}\label{thm:1awayclassificationP1XP1}
	Let $\cE$ be an initialized, indecomposable rank 2 vector bundle on $X=\p1 \times \p1$. Then, $\cE$ is special $1$-away ACM if and only if $\cE$ is one of the following:
	\begin{enumerate}[label=(\roman*)]
		\item $0 \longrightarrow \cE \longrightarrow \cO_{X}^{\oplus 2}(2h_1 +h_2) \oplus \cO_{X}^{\oplus 2}(h_1 +2h_2) \longrightarrow \cO_{X}^{\oplus 2}(2h_1 +2h_2) \longrightarrow 0 $
		\item $0 \longrightarrow \cE \longrightarrow \cO_{X} \oplus  \cO_{X}(h_1) \oplus \cO_{X}(h_2) \longrightarrow \cO_{X}(h_1 +h_2) \longrightarrow 0 .$
	\end{enumerate}
\end{theorem}
\begin{proof}
	First, assume that $\cE$ is special, initialized $1$-away ACM bundle of rank 2 on $X$. Then we have $s=0$ and $i_X =2$. So, by Theorem \ref{thm:GeneralBoundofk}, $-2 \leq k_0 \leq -1$ and $c=-2k_0 - 2$.\\
	\textbf{\textit{Case $k_0=-2$}}: Then $c=2$ and $\h2(\cE(t))=0 \; \text{for} \; t \geq -3$ by \ref{eqn:VanishingH0H2}. Since $\mathrm{h}^i (\cE(-1))=0$ for $0 \leq i \leq 2$, we have $c_2(\cE)=4$ by the equation \ref{eqn:RRochforp1Xp1}. Now, we know that we have a short exact sequence 
	\begin{equation}\label{ses:effectivedivisor}
		0 \longrightarrow \cO_{X}(-D) \longrightarrow \cO_{X} \longrightarrow \cO|_D \longrightarrow 0
	\end{equation}
	for any effect effective divisor $D$. Then, by the equation \ref{ses:effectivedivisor} and the Serre duality, we have $\mathrm{h}^{i}(\cE(ah_1+bh_2))=0$ for $i=0,2$ and for tuples $(a,b)=(-2,-2), (-3,-2), (-2,-3)$ and $(-3,-3)$. So, by the equation \ref{eqn:RRochforp1Xp1}, we have the following Table \ref{table4}
	\begin{table}[H]
		\centering
		\bgroup
		\def\arraystretch{1.5}
		\begin{tabular}{cccccc}
			\cline{2-5}
			\multicolumn{1}{c}{} &\multicolumn{1}{|c|}{0} & \multicolumn{1}{c|}{0} & \multicolumn{1}{c|}{0} &  \multicolumn{1}{c|}{0} & $q=3$ \\ 
			\cline{2-5}
			\multicolumn{1}{c}{} &\multicolumn{1}{|c|}{0} & \multicolumn{1}{c|}{2} & \multicolumn{1}{c|}{0}  & \multicolumn{1}{c|}{0} & $q=2$ \\ 
			\cline{2-5}
			\multicolumn{1}{c}{} &\multicolumn{1}{|c|}{0} & \multicolumn{1}{c|}{0} & \multicolumn{1}{c|}{2}  & \multicolumn{1}{c|}{2} & $q=1$ \\ 
			\cline{2-5}
			\multicolumn{1}{c}{} &\multicolumn{1}{|c|}{0} & \multicolumn{1}{c|}{0} & \multicolumn{1}{c|}{0} & \multicolumn{1}{c|}{0} & $q=0$ \\ 
			\cline{2-5}
			$(a,b,c)=$&$(-3,-3,1)$ & $(-2,-3,1)$ & $(-3,-2,0)$ &  $(-2,-2,0)$ & 
		\end{tabular}
		\egroup
		\caption{The values of $h^{q-c}\big(\cE(ah_1 + bh_2)\big)$ for $\p1 \times \p1$}
		\label{table4}
	\end{table}
	Then, by applying Proposition \ref{prop:monadonP1XP1} to $\cE(-2h)$, we have 
	$$
	0 \longrightarrow \cE \longrightarrow \cO_{X}^{\oplus 2}(2h_1 +h_2) \oplus \cO_{X}^{\oplus 2}(h_1 +2h_2) \longrightarrow \cO_{X}^{\oplus 2}(2h_1 +2h_2) \longrightarrow 0.
	$$
	\textbf{\textit{Case $k_0=-1$}}: Then $c=0$ and $\h2(\cE(th))=0 \; \text{for} \; t \geq -1$. By following exactly the same steps as in the case $k_0 =-2$, we have the following Table \ref{table5}
	\begin{table}[H]
		\centering
		\bgroup
		\def\arraystretch{1.5}
		\begin{tabular}{cccccc}
			\cline{2-5}
			\multicolumn{1}{c}{} &\multicolumn{1}{|c|}{1} & \multicolumn{1}{c|}{0} & \multicolumn{1}{c|}{0} &  \multicolumn{1}{c|}{0} & $q=3$ \\ 
			\cline{2-5}
			\multicolumn{1}{c}{} &\multicolumn{1}{|c|}{0} & \multicolumn{1}{c|}{1} & \multicolumn{1}{c|}{0}  & \multicolumn{1}{c|}{0} & $q=2$ \\ 
			\cline{2-5}
			\multicolumn{1}{c}{} &\multicolumn{1}{|c|}{0} & \multicolumn{1}{c|}{0} & \multicolumn{1}{c|}{1}  & \multicolumn{1}{c|}{1} & $q=1$ \\ 
			\cline{2-5}
			\multicolumn{1}{c}{} &\multicolumn{1}{|c|}{0} & \multicolumn{1}{c|}{0} & \multicolumn{1}{c|}{0} & \multicolumn{1}{c|}{0} & $q=0$ \\ 
			\cline{2-5}
			$(a,b,c)=$&$(-2,-2,1)$ & $(-1,-2,1)$ & $(-2,-1,0)$ &  $(-1,-1,0)$ & 
		\end{tabular}
		\egroup
		\caption{The values of $h^{q-c}\big(\cE(ah_1 + bh_2)\big)$ for $\p1 \times \p1$}
		\label{table5}
	\end{table}
	Then, by applying Proposition \ref{prop:monadonP1XP1} to $\cE(-h)$, we have 
	$$
	0 \longrightarrow \cE \longrightarrow \cO_{X} \oplus  \cO_{X}(h_1) \oplus \cO_{X}(h_2) \longrightarrow \cO_{X}(h_1 +h_2) \longrightarrow 0.
	$$
	
	Conversely, by applying Künneth’s theorem and considering its cohomology, one can check that kernels in items (i) and (ii) are special $1$-away ACM bundles of rank 2 on $\p1 \times \p1$.
\end{proof}
\begin{corollary}
	If $\cE$ is a special initialized $1$-away  ACM bundle of rank 2 on $\p1 \times \p1$, then it is weakly Ulrich.
\end{corollary}
\begin{proof}
	It is a direct consequence of Theorem \ref{thm:1awayclassificationP1XP1} in the light of Definition \ref{defn:weaklyUlrich}.
\end{proof}

\begin{theorem}\label{thm:hisconnectedquadric}
	Let $\cE$ be an indecomposable, initialized, special $\ell$-away ACM bundle of rank 2 on $X = \p1 \times \p1$. Then the module $\mathrm{H}_*^1 (\cE)$ is connected (hence $s=l-1$). 
\end{theorem}

\begin{proof}
	Assume on the contrary that $s \geq l$. Thanks to the Serre duality we can assume that there is a natural number $m \geq \left \lceil{\frac{s}{2}}\right \rceil$ such that $\h1 (\cE(k_0 +m))=0$. 
	
	If $\h2 (\cE(k_0 +m -1))=0$ then $\cE$ would be $k_0 + m +1$-regular. Then $\h1 (\cE (th)) = 0 $ for any $t \geq k_0 + m $, which is a contradiction (because $k_0 +s > k_0 + m$). So, $k_0 +m -1 < -c -1$ by \ref{eqn:VanishingH0H2}. So we have
	\begin{equation}\label{eqn:k+m+c}
		-c-1 \geq k_0 +m .
	\end{equation}
	Combining the inequality \ref{eqn:k+m+c} with the item (i) of Theorem \ref{thm:GeneralBoundofk} we have
	\begin{equation}\label{eqn:k+s-m+1}
		k_0 + m \geq 2m-s-1 \geq -1.
	\end{equation}

 So, by Proposition \ref{prop:Eisextension}, $\cE$ fits into 
	\begin{equation*}
		0 \longrightarrow \cO_{X}(D) \longrightarrow \cE  \longrightarrow \ \cO_X (ch-D) \longrightarrow 0
	\end{equation*}
	where $\h0 ( \cO_{X}(D)) \neq 0$ and $\h0 ( \cO_{X}(D-h))=0$. Then by Künneth’s theorem, $D=ah_1$ for some $a \geq 0$, without loss of generality. Since $\cE$ is indecomposable, $\mathrm{ext}^1 (\cO_X (ch-D) , \cO_{X}(D)) = \h1 (\cO_{X}(2D -ch)) \neq 0$. So $c \geq 2$ by Künneth’s theorem. Combining it with the inequality \ref{eqn:k+m+c}, we have $k_0 + m \leq -3$ contradicting the inequality \ref{eqn:k+s-m+1}. We deduce that $\mathrm{H}_*^1 (\cE)$ is connected.
\end{proof}

	\begin{theorem}\label{thm:2awayclassificationP1XP1}
		Let $\cE$ be an initialized, indecomposable rank 2 vector bundle on $X=\p1 \times \p1$. If $\cE$ is special $2$-away ACM then $\cE$ is one of the following:
		
		\begin{enumerate}[label=(\roman*)]
			\item $0 \longrightarrow \cE \longrightarrow \cO_{X}^{\oplus 3}(3h_1 +2h_2) \oplus \cO_{X}^{\oplus 3}(2h_1 +3h_2) \longrightarrow \cO_{X}^{\oplus 4}(3h_1 +3h_2) \longrightarrow 0 $
			\item $0 \longrightarrow \cE \longrightarrow \cO_{X}^{\oplus 5-c_2} (h)\oplus  \cO_{X}^{\oplus c_2 -2}(2h_1 + h_2) \oplus \cO_{X}^{\oplus c_2 -2}(h_1 + 2h_2) \longrightarrow \cO_{X}^{\oplus c_2 -1}(2h) \longrightarrow 0 $ \\
			where $2 \leq c_2 \leq 4$.
			\item $\cE$ is the cohomology of the monad\\
			$0 \longrightarrow \cO_{X}^{\oplus a} (-h) \oplus \cO_{X}^{\oplus b-d-c_2}(-h_2)\oplus \cO_{X}^{\oplus d}(-h_1) \longrightarrow \cO_{X}^{\oplus a+1-c_2} (-h) \oplus \cO_{X}^{\oplus b}(-h_2)\\
			\oplus \cO_{X}^{\oplus b}(-h_1) \oplus \cO_{X}^{\oplus a+1-c_2}\longrightarrow \cO_{X}^{\oplus d}(-h_2)\oplus \cO_{X}^{\oplus b-d-c_2}(-h_1) \oplus \cO_{X}^{\oplus a}\longrightarrow 0$\\
			where $a=\h1(\cE)$, $b=\h1(\cE(-h_1))$, $d=\h0(\cE(-h_1))$, $b-d \geq c_2 (\cE)$ and $a+1 \geq c_2 (\cE)$.
		\end{enumerate}
		Conversely, if $\cE$ is one of items (i) and (ii) then it is a special $2$-away ACM bundle of rank 2.
	\end{theorem}

	\begin{proof}
		Let $\cE$ be a special $2$-away ACM bundle of rank 2. Then by Theorem \ref{thm:hisconnectedquadric} $s=1$. Then $\det (\cE)= \cO ((-2k_0 -3)h)$ and $-3 \leq k_0 \leq -1$ by Theorem \ref{thm:GeneralBoundofk}. So, we have three cases for each value of $k_0$.
		
		For $k_0 = -3$ and $k_0 = -2$ we can follow same steps as in Theorem \ref{thm:1awayclassificationP1XP1} and get items (i) and (ii) respectively.
		
		For $k_0= -1$, unlike $k_0 = -3, -2$, we have non zero $\cC^{-1}$ module in Proposition \ref{prop:monadonP1XP1}. Therefore, we get a monad in item (iii). In this case, we are unable to get converse statement.
	\end{proof}
	\begin{corollary}
		If $\cE$ is an initialized, indecomposable, special $2$-away  ACM bundle of rank 2 on $\p1 \times \p1$, then $\cE$ is weakly Ulrich up to shifts.
	\end{corollary}
	\begin{proof}
		If $\cE$ is as in items (i), (ii) and (iii) of Theorem  \ref{thm:2awayclassificationP1XP1} then $\cE(-1)$,  $\cE$ and $\cE(1)$ is weakly Ulrich respectively.
	\end{proof}
Now, for any positive integer $\ell$, we can construct indecomposable $\ell$-away ACM bundles of rank 2 on $ \p1 \times \p1$ using line bundles in Theorem \ref{thm:lawayLineBundlesonQuadric}.
	\begin{theorem}\label{thm:generalexistencerank2quadric}
	There exist an indecomposable, special, initialized $\ell$-away ACM bundle $\cE$ of rank 2 on $X = \p1 \times \p1$ obtained from the extension
	\begin{equation}\label{ses:extensionoflinebundlesonquadric}
	0 \longrightarrow \cO_X((\ell+1)h_1) \longrightarrow \cE \longrightarrow \cO_X((\ell+1)h_2) \longrightarrow 0.
	\end{equation}
	 Also, the point corresponding to $\cE$ in the moduli space of simple sheaves on $\p1 \times \p1$ is smooth and lies in a component of dimension $2\ell^2 +2\ell -3$.
	\end{theorem}

\begin{proof}
First note that 
$$\mathrm{ext}^1(\cO_X((\ell+1)h_2), \cO_X((\ell+1)h_1)) = \mathrm{h}^1 (\cO_X((\ell+1)h_1 -(\ell+1)h_2))=\ell^2 +2\ell > 0$$
 by Künneth’s theorem. So, there exists $\cE \in \mathrm{Ext}^1 \big( \cO_X((\ell+1)h_2), \cO_X((\ell+1)h_1) \big)$ which is not isomorphic to $\cO_X((\ell+1)h_1) \oplus  \cO_X((\ell+1)h_2)$.\\
 Then, one can easily show that $\h0 (\cE)= \h0(\cO_X((\ell+1)h_1)) + \h0 (\cO_X((\ell+1)h_2))= 2\ell+4$ by considering cohomology sequence of \ref{ses:extensionoflinebundlesonquadric} because $\h1 (\cO_X((\ell+1)h_1)) = 0$. By a similar argument $\h0 (\cE(-h))=0$.\\
  For $-\ell-1 \leq t \leq -2$, $\h0(\cO_X(th_1+(\ell+1+t)h_2))=0$. So
 \[\h1 (\cE(t))\geq \h1(\cO_X((\ell+1+t)h_1+th_2)) > 0.
 \]
 For $t<-\ell-1$ and $t>-2$, both $\h1(\cO_X((\ell+1+t)h_1+th_2))$ and $\h1(\cO_X(th_1+(\ell+1+t)h_2))$ are 0. So $\h1(\cE(t))=0$ for these values of $t$.\\
Obviously, $\det (\cE)= \cO \big( (\ell+1)h_1 +(\ell+1)h_2 \big) = \cO ((\ell+1)h ) $. Therefore, $\cE$ is a special, initialized $\ell$-away ACM bundle of rank 2 on $\p1 \times \p1$.

Now, to show $\mu$-semistability of $\cE$, we apply \cite[Corollary 4]{JMPS17}: we need to show that $\h0 \big( \cE(ah_1 + bh_2) \big)=0$ for all $a+b < -\ell-1$. When we tensor the short exact sequence \ref{ses:extensionoflinebundlesonquadric} with $\cO_{X}(ah_1 + bh_2)$ and then consider the cohomology sequence, one has 
$$\h0 \big( \cE(ah_1 + bh_2) \big) \leq \h0 \big( \cO_{X} ((a+\ell+1)h_1 + bh_2) +\h0 \big( \cO_{X} (ah_1 + (b+\ell+1)h_2).
$$
The right hand side is zero because of Künneth’s theorem and inequality  $a+b < -\ell-1$. Therefore, $\cE$ is $\mu$-semistable.

If $\cE$ were decomposable, $\cE = \cO_{X}(D) \oplus \cO_{X} ((\ell+1)h -D)$ where $D = ah_1 + bh_2$ for some integers $a$ and $b$. Since $\cE$ is initialized, we have both $a-1 <0$ and $\ell-a < 0$, which is a contradiction. So $\cE$ is indecomposable.

Since $\cE$ is indecomposable $\mu$-semistable, $\cE(-\lceil \frac{\ell+1}{2} \rceil)$ is also indecomposable $\mu$-semistable and $\det \big(\cE(-\lceil \frac{\ell+1}{2} \rceil) \big) = \cO(kh)$ where $ k \in \{-1, 0\}$. Also, since $\cE$ is initialized, $(1+k) \h0 (\cE(-\lceil \frac{\ell+1}{2} \rceil)) = 0$. Therefore, by \cite[Proposition 3.2]{ACG21}, $\cE(-\lceil \frac{\ell+1}{2} \rceil)$ is a simple vector bundle. Hence, $\cE$ is simple.

To show that $\cE$ corresponds to a smooth point in the moduli space of simple sheaves (see \cite{AK80}), it is enough to show that $\mathrm{ext}^2 \big(\cE, \cE \big)= \h2 \big(\cEv \otimes \cE \big)=0$. To see this, we tensor the short exact sequence \ref{ses:extensionoflinebundlesonquadric} with $\cEv$ and then consider the cohomology sequence. Then we obtain 
\begin{align*}
\h2 \big(\cEv \otimes \cE \big) &\leq \h2 \big(\cEv ((\ell+1)h_1) \big) + \h2 \big(\cEv ((\ell+1)h_2) \big)\\
& \leq \h0 \big(\cE (-(\ell+3)h_1-2h_2) \big) + \h0 \big(\cE (-2h_1-(\ell+3)h_2) \big).
\end{align*}
Since $\cE$ is $\mu$-semistable and $(-(\ell+3)h_1-2h_2)h < -c_1(\cE)(h/2)$, by Hoppe's Criterion (see \cite[Corollary 4]{JMPS17}), $\h0 \big(\cE (-(\ell+3)h_1-2h_2) \big)=0$. Similarly, $ \h0 \big(\cE (-2h_1-(\ell+3)h_2) \big)=0$. So, $\h2 \big(\cEv \otimes \cE \big)=0$ and this is what we wanted to prove.

Lastly, the dimension of the moduli space is equal to $\h1 \big(\cEv \otimes \cE \big)$. Since $\cE$ is simple and $\h2 \big(\cEv \otimes \cE \big)=0$, $\h1 \big(\cEv \otimes \cE \big)= 1-\chi(X, \cE \otimes \cEv)$. Since we have the short exact sequence \ref{ses:extensionoflinebundlesonquadric}, $\chi( \cE \otimes \cEv)= \chi(\cEv ((\ell+1)h_1)) + \chi(\cEv ((\ell+1)h_2))$. By the help of Equation \ref{eqn:RRochforp1Xp1}, we can find $\chi(\cEv ((\ell+1)h_1)) = \chi( \cEv ((\ell+1)h_2)) = -\ell^2 -\ell +2$. Therefore, $\h1 \big(\cEv \otimes \cE \big)= 2\ell^2 + 2\ell -3$.
\end{proof}

\begin{corollary}\label{remark:supernaturalquadric}
	The bundle $\cE$ constructed in Theorem \ref{thm:generalexistencerank2quadric} is a supernatural bundle and its Hilbert polynomial has integral roots  $-\ell-2$ and $-1$.
\end{corollary}
\begin{proof}
 We have shown in the proof of Theorem \ref{thm:generalexistencerank2quadric} that
\begin{eqnarray*}
	\h0(\cE(th)) &=& 0 \text{ for } t \leq -1\\
	\h1(\cE(th)) &=& 0 \text{ for } t \leq - \ell -2 \text{ and } t \geq -1\\
	\h2(\cE(th)) &=& 0 \text{ for } t \geq - \ell -2.
\end{eqnarray*}
\end{proof}
Finally, we can give a construction for $\ell$-away ACM bundles of higher rank on $\p1 \times \p1$ for any positive integer $\ell$.
\begin{theorem}\label{thm:higherrankquadric}
	There exist simple $\ell$-away ACM bundle $\cF$ on $\p1 \times \p1$ of any even rank $r=2m$ and for any positive number $\ell$. Also, the point corresponding to $\cF$ in the moduli space of simple sheaves on $\p2$ is smooth and lies in a component of dimension $m^2(2\ell^2+4\ell-2)+1$.
\end{theorem}

\begin{proof}
	Let $\cE$ be the bundle in the short exact sequence \ref{ses:extensionoflinebundlesonquadric}. We will follow the same steps as in Theorem \ref{thm:higherrankp2}. Then by Theorem \ref{thm:bondal} (see also \cite[Section 6]{M17}), $\mathcal{E}$ is mapped to the complex of quiver representation
\[
\begin{tikzcd}[column sep=tiny]
	&&R\mbox{Hom}(\mathcal{O}(2h_1+h_2),\mathcal{E})\arrow[shift left]{dr}\arrow[shift right,swap]{dr} &\\
	&R\mbox{Hom}(\mathcal{O}(2h_1+2h_2),\mathcal{E})\arrow[shift left]{ur}\arrow[shift right,swap]{ur}\arrow[shift left]{dr}\arrow[shift right,swap]{dr}&&R\mbox{Hom}(\mathcal{O}(h_1+h_2),\mathcal{E})\\
	&&R\mbox{Hom}(\mathcal{O}(h_1+2h_2),\mathcal{E})\arrow[shift left]{ur}\arrow[shift right,swap]{ur}&
\end{tikzcd}.
\]
	The right vertex vanishes because $\Ext^i \big((\mathcal{O}(h_1+h_2), \mathcal{E} \big) = \mathrm{H}^{i} \big(\cE (-h) \big)=0$ for all $i$ by Corollary \ref{remark:supernaturalquadric}. Due to the same steps as in Theorem \ref{thm:generalexistencerank2quadric}, the graded vector spaces at the left and the middle vertices are concentrated in degree 1. Therefore, we have the following quiver
\[
\begin{tikzcd}
	& 1 & \\
	0\arrow[shift left]{ur}\arrow[shift right,swap]{ur}\arrow[shift left]{dr}\arrow[shift right,swap]{dr} & & \\
	& 2 &
\end{tikzcd}
\]
	with the representation of dimension vector $\underline{d} = (d_0, d_1, d_2)$ where $d_0 = \h1 (\cE (-2h_1-2h_2))$, $d_1 = \h1 (\cE (-2h_1-h_2))$ and $d_2 = \h1 (\cE (-h_1-2h_2))$. By Theorem \ref{thm:generalexistencerank2quadric}, we have $d_0 =2\ell $, $d_1 =\ell+1$ and $d_2 = \ell+1$.\\
	Then,  by Proposition \ref{prop:eulerdimensionvector}, we have
	 $$\chi(\underline{d}, \underline{d})= d_0^2 + d_1^2 + d_2^2 -2d_0d_1 -2d_0d_2 = -2\ell^2 -4\ell + 2 < 0.$$ 
	
	Then, one can follow the same steps in Theorem \ref{thm:higherrankp2} and see that the general element in $Rep(m \underline{d})$ corresponds to an indecomposable $\ell$-away ACM bundle.
	
	Lastly, notice that $\mathrm{ext}^2 (\cF, \cF) =0$ by Lemma \ref{lem:simpleext2vanish} since $\cF$ is simple. So, the point corresponding to $\cF$ in the moduli space of simple sheaves on $\p1 \times \p1$ is smooth and its dimension can be computed by considering the Euler form as
	\[
	1 - \chi(m \underline{d}, m \underline{d}) = 1 - m^2 \chi(\underline{d}, \underline{d})=m^2(2\ell^2+4\ell-2)+1.
	\]
\end{proof}
\begin{corollary}\label{cor:higherrankweaklyUlrichP1XP1}
	There exist weakly Ulrich bundles of any even rank on $\p1 \times \p1$.
\end{corollary}
\begin{proof}
Let $\cF$ is a 2-away ACM bundle of rank $r=2m$ constructed in Theorem \ref{thm:higherrankquadric}. Then $\cF(-1)$ is weakly Ulrich.
\end{proof}
\section{$\ell$-away ACM Bundles on Blow up of $\p2$ up to three non collinear points}\label{sec:blowup}
Let $X$ be the blow up of $\p2$ up to three non collinear points with polarization $h = - K_X $. In this section, we will study $\ell$-away ACM line bundles on $X$. Then, we will use these line bundles to construct $\ell$-away ACM bundles of rank 2. 

Let us call $m$ for the pull-back of a line in $\p2$ via blow up map, and $e_i$ for exceptional divisors on $X$.  The canonical bundle $K_{X}$ is $-3m+\sum_i e_i=-h$. For a line bundle $\ls=\oo(am+\sum b_ie_i)$ on $X$, we have
$$\chi(\ls)={a+2\choose 2}+\sum_i \frac{b_i(1-b_i)}{2}={a+2\choose 2}+\sum_{b_i\geq 0} \frac{b_i(1-b_i)}{2}-\sum_{b_i<0} {d_i+1\choose 2}$$  
where $d_i$ stands for $-b_i$ when $b_i<0$ (it is zero if $b_i\geq 0$).

\begin{proposition}\label{blumenhagen}
	Let $X$ be the blow up of $\p2$ up to three non collinear points $p_1,p_2,p_3$ and $\ls=\oo(am+\sum b_ie_i)$ be a line bundle on $X$. Then, for $a\geq -2$ we have
	$$
	\operatorname{h}^i (\ls) = \left\{ \begin{array}{ll}
		0 & i=2\\
		A-\chi(\ls) & i=1\\
		A & i=0
	\end{array} \right.
	$$
	where $A$ is the dimension of the space of forms of degree $a$ in three variables and with zeroes of order at least $d_i$ at the points $p_i$.
\end{proposition}
\begin{proof}
	By the Serre duality, we have
	$$\h2 (\cO_X(am + \sum b_i e_i)) = \h0 (\cO_X((-a-3)m - \sum ((b_i -1) e_i)).$$
	If $a\geq -2$ then the latter dimension is zero because $(-a-3)m - \sum ((b_i -1) e_i)m < 0$ and $\cO_{X}(m)$ is base point free. The assertion for $\h0$ follows from the definition of $X$ because if $b_i > 0$, it is immediate to check that $b_i e_i$ is a fixed component. At this point the computation of $\h1$ is
	trivial.
\end{proof}
By choosing the proper coordinates, we can assume that $p_1=[1:0:0],\ p_2=[0:1:0]$ and $\ p_3=[0:0:1].$ Therefore, to calculate $A$, we can focus on monomials of order $a$ as they form a basis of the space of degree $a$ forms. Also, a form $f$ has order at least $d_i$ at $p_i$ if and only if every monomial that contributes to $f$ has at least order $d_i$. So we obtain:
\begin{proposition}\label{mo}
	Let $X$ be the blow-up $\p2$ up to $1\leq n\leq 3$ non collinear points and $\ls=\oo_X(am+\sum b_ie_i)$ be a line bundle on  $X$. If $a\geq -2$, then we have $A=A_n$ where
	\begin{align*}
		A_1=max&\left\{{a+2\choose 2}-{d_1+1\choose 2},0\right\}\\
		A_2=max&\left\{{a+2\choose 2}-{d_1+1\choose 2}-{d_2+1\choose 2}+{d_1+d_2-a\choose 2},0\right\}\\
		A_3=max&\Big\{{a+2\choose 2}-{d_1+1\choose 2}-{d_2+1\choose 2}-{d_3+1\choose 2}+{d_1+d_2-a\choose 2}+\\
		&+{d_2+d_3-a\choose 2}+{d_3+d_1-a\choose 2}-{d_1+d_2+d_3-2a-1\choose 2},0\Big\} .
	\end{align*}
\end{proposition}
\begin{proof}
	Let $x,y,z$ be coordinates in $\p2$.
	
	\textbf{\textit{Case $n=1$:}} The monomial $f$ does not have a zero of order $d_1$ in $p_1$ if and only if $f=gx^{a-d_1+1}$. We can choose $g$ arbitrary and we have ${deg(g)+2\choose 2}={d_1+1\choose 2}$ many ways to do so. Since the number of all monomials is ${a+2\choose 2}$, we obtain our formula.
	
	\textbf{\textit{Case $n=2$:}} As we showed earlier, we have ${d_1+1\choose 2}$ many monomials that vanish at $p_1$ with order smaller than $d_1$. This is the same for $p_2$. But, there can happen that some monomials lie in both of these bases. Such a case occurs for monomials of the form $x^{a-d_1+1}y^{a-d_2+1}g$. Since we have ${d_1+d_2-a\choose 2}$ many of them, we obtain 
	$$A=max\left\{{a+2\choose 2}-{d_1+1\choose 2}-{d_2+1\choose 2}+{d_1+d_2-a\choose 2},0\right\} .$$
	
	\textbf{\textit{Case $n=3$:}} An analogous reasoning as in the previous case gives the desired formula.
\end{proof}
Now we can characterize $\ell$-away ACM line bundles on the blow up of $\p2$ at a single point.
\begin{theorem}\label{rank 1 on blow up in one}
	Let $X$ be the blow up of $\p2$ at a single point. A line bundle $\ls=\oo_X(am+b_1e_1)$ on $X$ is initialized $\ell$-away ACM if and only if $a$ and $b_1$ are as follows:
	\begin{itemize}
		\item $a=2 \ell +2 \pmod{3}$ and $b_1=\frac{2 \ell +2-a}{3}$
		\item $a=2 \ell \pmod{3}$ and $b_1=\frac{2 \ell -a}{3}+1$
		\item $a= \ell +1$ and $b_1=-( \ell +1)$
		\item $a= \ell +3$ and $b_1=-( \ell +2)$
	\end{itemize}
	Moreover, $\h1(\ls(th))>0$ if and only if $t\in I$ where $I$ is as follows:
	\begin{itemize}
		\item $\left[\frac{-( \ell +1)-a}{3},\frac{2( \ell +1)-a}{3}-2\right]$
		\item $\left[\frac{-a- \ell}{3},\frac{2 \ell -a}{3}-1\right]$
		\item $[-\ell,-1]$
		\item $[-\ell -1,-2]$
	\end{itemize}
\end{theorem}
We precede the proof with a couple of technical lemmas. In the beginning, we will establish the conditions for the vanishing of intermediate cohomology of a fixed line bundle $\cM.$
\begin{lemma}\label{tl1}
	For a line bundle $\cM=\oo_X(am+b_1e_1)$, $\h1(\cM)=0$ if and only if one of the followings holds
	\begin{itemize}
		\item $a>-2$ and $-(a+1)\leq b_1\leq 1$
		\item $a\leq -2$ and $-(1+(-3-a))\leq 1-b_1\leq 1.$
	\end{itemize}
\end{lemma}
\begin{proof}
	We analyze three cases.
	
	\textbf{\textit{Case $a>-2$ and $b_1\geq 0$}:} We obtain $A={a+2\choose 2}$ and $\chi (\cM)={a+2\choose 2}-\frac{b_1(1-b_1)}{2}$. Therefore
	$$\h1 (\cM)=A-\chi (\cM)=\frac{(b_1-1)b_1}{2} .$$
	So $\h1 (\cM)=0$ if and only if $0\leq b_1\leq 1.$
	
	\textbf{\textit{Case $a>-2$ and $b_1<0$}:} Recall that $d_1=-b_1.$ If $d_1>a+1$ then $A=0$ and 
	$$\h1 (\cM)=-\chi (\cM)={d_1+1\choose 2}-{a+2\choose 2}>0.$$ 
	For $d_1\leq a+1$ we have $A={a+2\choose 2}-{d_1+1\choose 2}=\chi (\cM)$. Thus $\h1 (\cM)=0$ identically. So, $\h1 (\cM)=0$ if and only if $d_1\leq a+1$ or equivalently $b_1\geq -(a+1)$. Joining this with the inequality from the previous case we obtain 
	$$-(a+1)\leq b_1\leq 1.$$
	
	\textbf{\textit{Case $a\leq -2$}:} This part follows from the previous cases after applying the Serre duality.
\end{proof}
Let us continue and find conditions on $t$ to have $\h1(\cM(th ))\neq 0$. 
\begin{lemma}
	For a line bundle $\cM=\oo_X(am+b_1e_1)$ the following holds
	$$\h1(\cM(th))\neq 0\iff t\in\left(min\left\{b_1,\frac{-a-b_1-1}{2}\right\},max\left\{b_1-1,\frac{-a-b_1-1}{2}\right\}\right) .$$
	As a result, $\H1_* (\cM)$ is connected for any line bundle on $X$.
\end{lemma}
\begin{proof}
	Recall that $h=3m-e_1$, thus $\cM(t h)=\oo_X((a+3t)m+(b_1-t)e_1)$. Then, $\h1(\cM (th))\neq 0$ if and only if  inequalities from Lemma \ref{tl1} do not hold for $\cM(th)$. So we have two cases:
	
	\textbf{\textit{Case $a+3t>-2$}}: 
	First inequality from Lemma \ref{tl1} must be invalidated, so
	$$b_1-t>1\ \ \text{or}\ \ -(a+3t+1)>b_1-t.$$
	Transforming this, we obtain
	$$b_1-1>t\ \ \text{or}\ \ \frac{-a-b_1-1}{2}>t.$$
	Joining this with $a+3t>-2$, we get
	$$t\in\left(\frac{-2-a}{3},max\left\{b_1-1,\frac{-a-b_1-1}{2}\right\}\right) .$$
	
	\textbf{\textit{Case $a+3t\leq -2$}}:
	Analogous reasoning leads to
	$$t>b_1\ \ \text{or}\ \ t>\frac{-a-b_1-1}{2} .$$
	Therefore
	$$t\in\left(min\left\{b_1,\frac{-a-b_1-1}{2}\right\},\frac{-2-a}{3}\right] $$
	and the analysis of this case is complete.\par
	Putting these two cases together, we obtain
	$$t\in\left(min\left\{b_1,\frac{-a-b_1-1}{2}\right\},\frac{-2-a}{3}\right]\cup \left(\frac{-2-a}{3},max\left\{b_1-1,\frac{-a-b_1-1}{2}\right\}\right) .$$
	We will check below that we can get rid of the fraction $\frac{-2-a}{3}$ in the middle and write
	$$t\in\left(min\left\{b_1,\frac{-a-b_1-1}{2}\right\},max\left\{b_1-1,\frac{-a-b_1-1}{2}\right\}\right)=I .$$
	Indeed, it's obvious that $I$ is a subset of the prior union of open intervals. On the other hand, if, for an integer $x$, we have
	$$x\in \left(\frac{-2-a}{3},max\left\{b_1-1,\frac{-a-b_1-1}{2}\right\}\right) \backslash \ I$$
	then $\frac{-2-a}{3}<x\leq min\{b_1,\frac{-a-b_1-1}{2}\}$. Hence $x\leq b_1$ and $2x\leq -a-b_1-1$. After summing these two inequalities we get
	$$3x\leq -a-1.$$
	But the first inequality $\frac{-2-a}{3}<x$ implies
	$$3x>-2-a .$$
	Since $3x$ is an integer, $3x=-a-1$. Therefore $x=b_1$ and $x=\frac{-a-b_1-1}{2}$. But then $x\geq max\left\{b_1-1,\frac{-a-b_1-1}{2}\right\}$ which gives us a contradiction. Analogously, for the other component of the union of intervals, we have a similar result. Having this we can finally summarize our calculations as
	$$\h1(\cM(th))\neq 0\iff t\in\left(min\left\{b_1,\frac{-a-b_1-1}{2}\right\},max\left\{b_1-1,\frac{-a-b_1-1}{2}\right\}\right)$$ 
\end{proof}	
From the previous lemma, we can see that $\cM$ is $\ell$-away if and only if the interval $I$ contains $\ell$ many integers. We will deal below with the restrictions on $a$ and $b_1$ such that $|\mathbb{Z}\cap I|= \ell.$
\begin{lemma}\label{tl3}
	Let $\cM=\oo_X(am+b_1e_1)$ be a line bundle on $X$. Then $\cM$ is $\ell$-away ACM if and only if $|a+3b_1|\in\{2 \ell+2,2 \ell+3\}.$ Also, $\cM$ is ACM if and only if $|a+3b_1|<4.$
\end{lemma}
\begin{proof}
	Again we have two cases depending on the form of $I$ ($I$ defined as in the previous lemma). 
	
	\textbf{\textit{Case $I=\left(b_1,\frac{-a-b_1-1}{2}\right)$}}:
	$$|\mathbb{Z}\cap I|= \ell \iff b_1+ \ell < \frac{-a-b_1-1}{2}\leq b_1+ \ell +1 .$$
	This implies that
	$$-(2 \ell +3) \leq 3b_1+a < -(2 \ell +1).$$
	
	\textbf{\textit{Case $I=\left(\frac{-a-b_1-1}{2},b_1-1\right)$}}:
	$$|\mathbb{Z}\cap I|= \ell \iff b_1-(\ell +2) \leq \frac{-a-b_1-1}{2}< b_1-(\ell +1) .$$
	Therefore, 
	$$2 \ell +1 < 3b_1+a \leq 2 \ell +3.$$
	So, $\cM=\oo_X(am+b_1e_1)$ is $\ell$-away ACM if and only if $|a+3b_1| \in \{2 \ell +2,2 \ell+3\}.$
	
	The second part of the statement is the direct result of the first part. In fact, ACM bundles on anticanonically polarized Fano surfaces have been studied in \cite{PT09}.
\end{proof}
Finally, we can prove our theorem. Using Lemma \ref{tl3} the only thing left is to deal with the restrictions on $a$ and $b_1$ such that $\ls$ is initialized.
\begin{proof}[Proof of Theorem \ref{rank 1 on blow up in one}]
	$\ls$ is initialized if and only if
	$$\h0(\ls)>0\ \ \text{and} \ \ \h0(\ls(-h))=0 .$$
	But 
	$$\h0(\ls)=A=max\left\{{a+2\choose 2}-{d_1+1\choose 2},0\right\}$$
	and this is nonzero only when $a\geq 0$ and $d_1<(a+1)$; that is, $b_1>-(a+1).$ Since $\h0(\ls(-h))=0$, we have
	$$a-3<0\ \ \ \text{or}\ \ \ b_1+1\leq -(a-3+1)=-(a-2) .$$
	We obtain two cases:
	\begin{align}
		3>a\geq0\ \ &\text{and}\ \ b_1>-a-1 \label{cfi1} \\
		a\geq0\ \ &\text{and}\ \ -a+1\geq b_1>-a-1 \label{cfi2}.
	\end{align}
	In the first case \ref{cfi1}, we have
	$$0\leq a\leq 2$$
	$$b_1>-(a+1)\implies a+3b_1\geq -2a\geq -4 .$$
	For $ \ell =1$, it may happen that
	$$a=2,\ \ \ b_1=-2 .$$
	Then $|a+3b_1|=4$, and so $\ls=\oo_X(2m-2e_1)$ is an initialized 1-away ACM line bundle. This is the only possibility for $\ls$ to be initialized $\ell$-away ACM with $a+3b_1<0$. When $a+3b_1>0$, we have
	$$a+3b_1=2 \ell +2\ \ \ \text{or}\ \ \ a+3b_1=2 \ell +3 .$$
	Since $0\leq a\leq 2$, we have two solutions:
	$$a=2 \ell+2 \ (\operatorname{mod} 3) \ \ \ and\ \ \ b_1=\frac{2 \ell+2-a}{3}$$
	$$a=2 \ell+3 \ (\operatorname{mod} 3) \ \ \ and\ \ \ b_1=\frac{2 \ell+3-a}{3} .$$
	In the second case \ref{cfi2}, we have
	$$-a\leq b_1\leq -a+1 .$$
	This gives us two solutions:
	$$b_1=-a$$
	$$a+3b_1=-2a=-(2 \ell +2) .$$
	The minus sign is here because $a\geq 0.$
	Therefore, we have
	$$a= \ell +1\ \  \text{and} \ b_1=-(\ell+1).$$
	Similarly, for the second solution, we have
	$$b_1=-a+1$$
	$$a+3b_1=-2a+3=-(2 \ell+3)$$
	$$a=\ell+3,\ \ \ b_1=-(\ell+2) .$$
	Inserting the calculated values of $a$ and $b_1$ into the formula for $I$ we get the "moreover" part.
\end{proof}
We can now use the Theorem \ref{rank 1 on blow up in one} to construct $\ell$-away ACM bundles of rank 2.
\begin{theorem}
	Let $X$ be the blow up of $\p2$ at a single point. Then, there exists an indecomposable, initialized $\ell$-away ACM vector bundle of rank 2 on $X$.
\end{theorem}
\begin{proof}
	We will prove this theorem separately for $\ell=1$ and $\ell \geq 2.$ 
	
	For $\ell=1$, consider the following exact sequence:
	\begin{equation}\label{ses:blowonepointl1}
		0\to \oo_X(e_1)\to\cE\to \oo_X(2m-2e_1)\to0 .
	\end{equation}
	Denote $\oo_X(2m-2e_1)$ by $\ls$ and $\oo_X(e_1)$ by $\cM$. First notice that
	\begin{align*}
		\text{ext}^1(\ls,\cM)&=\h1(\cM\otimes  \ls^\vee)=\h1(\oo_X(-2m+3e_1)) .
	\end{align*}
	Applying Proposition \ref{mo} with $a=-2$ and $b_1=3$ (thus $d_1=0$) we get that $A = \h1(\oo_X(-2m+3e_1))=0.$ Also $\chi(\oo_X(-2m+3e_1))=-3.$ Finally by Proposition \ref{blumenhagen} we get that 
	$$\h1(\oo_X(-2m+3e_1))=A-\chi(\oo_X(-2m+3e_1))=0-(-3)=3>0$$
	so, there exists a non split extension. From Theorem \ref{rank 1 on blow up in one} we get that $\ls$ is initialized and furthermore $\h1(\ls(th))\neq 0$ only for $t=-1$. On the other hand, $\cM$ is initialized because it satisfies Inequality (\ref{cfi1}) and it is ACM by Corollary \ref{cor5}. Since $\ls$ and $\cM$ are initialized, $\cE$ is also initialized. Because $\h1(\ls(-h))\neq 0$ and $\h2(\cM(-h))=0$ thus $\h1(\cE(-h))\neq 0$ and so $\cE$ is initialized $1$-away ACM.
	
	For $\ell \geq 2$ we take:
	\begin{equation}
		0\to \oo_X(\ell m-\ell e_1)\to\cE\to \oo_X(2e_1)\to 0 .
	\end{equation}
	Denote $\oo_X(\ell m-\ell e_1)$ by $\cM$ and $\oo_X(2e_1)$ by $\ls$. We have:
	$$\text{ext}^1(\ls,\cM)=\h1(\cM\otimes \ls^\vee)=\h1(\oo_X(\ell m-(\ell +2)e_1)) .$$
	Applying Proposition \ref{mo} with $a=\ell$ and $b_1=-(\ell+2)$ (thus $d_1=\ell +2$) we get that $A=\h1(\oo_X(\ell m-(\ell +2)e_1))=0.$ Also $\chi(\oo_X(\ell m-(\ell +2)e_1))={\ell+2 \choose 2}-{\ell+3 \choose 2}<0$ hence by Proposition \ref{blumenhagen} $\h1(\oo_X(\ell m-(\ell +2)e_1))>0$ and so there exist a non split extension. From Theorem \ref{rank 1 on blow up in one} we have that $\cM$ is initialized and $\h1(\cM(th))\neq 0$ if and only if $t\in [-(\ell -1),-1].$ Similarly  by Theorem \ref{rank 1 on blow up in one} $\ls$ is initialized and $\h1(\ls(th))\neq 0$ if and only if $t\in [-1,0].$ From the long exact sequence of cohomologies it follows that $\h1(\cE(th))\neq 0$ exactly for $t\in [-(\ell -1),0]$ hence $\cE$ is initialized $\ell$-away ACM bundle. 
	
	Finally, let us prove the indecomposability of constructed bundles. For $\ell \geq 2$, assume on the contrary that $\cE=\oo_X(am+be_1)\oplus\oo_X(a'm+b'e_1)$. Since $\cE$ is initialized, at least one of these bundles must be initialized too. We can assume without losing the generality that it is the first one. Therefore, from previous calculations, we obtain
	\begin{equation}\label{cases}
		a\in\{0,1,2\}\ \ \text{or}\ \ b=-a\ \ \text{or}\ \ b=-a+1.
	\end{equation}
	Comparing Chern classes we must have
	$$a+a'=\ell,\ \ \ \ b+b'=2-\ell$$
	$$aa'-bb'=2.$$
	Hence
	\begin{equation}\label{ziem}
		a(\ell -a)+b(\ell +b-2)=2\ell.
	\end{equation}
	We have five cases to check by \ref{cases}. If we try to solve equation \ref{ziem} for each one of these five cases, we will have a contradiction. By a similar analysis, we will have the same result for the case $\ell =1$. Therefore, the constructed bundles are indecomposable.
\end{proof}
To show the existence of $\ell$-away ACM bundles of rank 2 on the blow ups at two distinct and three non collinear points we use the following construction.
\begin{theorem}\label{exm23}
	Let $X$ be the blow up of $\p2$  at two distinct or three non collinear points and $\cM=\oo_X((\ell +1)m-(\ell +1)e_1-e_2)$. Then $\cM$ is an initialized $\ell$-away ACM line bundle and $\h1(\cM(th))\neq 0$ for $-\ell \leq t\leq -1.$
\end{theorem}
Before proving this theorem, we will show the conditions for the vanishing of the middle cohomology of an arbitrary line bundle.
\begin{lemma}\label{tlo}
	Let $X$ be the blow up of $\p2$  at three non collinear points. For a line bundle $\ls=\oo_X(am+b_1e_1+b_2e_2+b_3e_3)$, $\h1(\ls)=0$ if and only if one of the following holds:
	\begin{itemize}
		\item $a>-2$, $b_i\leq 1$ and $-a-1\leq min\{b_i,b_i+b_j\}$
		\item $a\leq -2$, $0\leq b_i$ and $max\{b_i,b_i+b_j-1\}\leq -(a+1)$
	\end{itemize}
	where $i,j$ go through all possible pairs from the set $\{1,2,3\}$ such that $i\neq j$.  
\end{lemma}
\begin{proof}  
	By the Serre duality, it's enough to prove the lemma for $a>-2$. The proof will follow by case-by-case analysis.
	
	\textit{Case 1: $a>-2$ and $b_1,b_2,b_3\geq 0$}:\\
	$$\h1 (\ls)=A-\chi (\ls)=\frac{(b_1-1)b_1}{2}+\frac{(b_2-1)b_2}{2}+\frac{(b_3-1)b_3}{2} .$$
	Thus, $\h1 (\ls)=0$ if and only if $0\leq b_1,b_2,b_3\leq 1.$
	
	\textit{Case 2: $a>-2,\ b_2,b_3\geq 0$ and $b_1<0$}:\\
	If $d_1>a+1$ then $A=0$ and so 
	$$\h1 (\ls)= -\chi (\ls) =\frac{(b_2-1)b_2}{2}+\frac{(b_3-1)b_3}{2}+{d_1+1\choose 2}-{a+2\choose 2}>0.$$ 
	For $d_1\leq a+1$ we have $A={a+2\choose 2}-{d_1+1\choose 2}$. In this case, we get
	$$A-\chi (\ls)=\frac{(b_2-1)b_2}{2}+\frac{(b_3-1)b_3}{2}.$$ This is zero if and only if $b_2,b_3\leq 1$
	
	\textit{Case 3: $a>-2$, $b_1,b_2<0$ and $b_3\geq 0$}:\\
	From Proposition \ref{mo} we get that
	$$A=max\left\{Y+Z,0\right\}$$
	where 
	\begin{align*}
		Y&={a+2\choose 2}-{d_1+1\choose 2}-{d_2+1\choose 2},\\
		Z&={d_1+d_2-a\choose 2}+{d_1-a\choose 2}+{d_2-a\choose 2}-{d_1+d_2-2a-1\choose 2}.
	\end{align*}
	Notice that $Z\geq 0$ since $d_1+d_2-a\geq d_1+d_2-2a-1$. Now, if $Y+Z<0$ then $Y<0$ and
	$$\h1(\ls)= -\chi (\ls) =-Y+\frac{(b_3-1)b_3}{2}>0.$$
	Otherwise, $A=Y+Z$ and 
	$$\h1(\ls)=Z+\frac{(b_3-1)b_3}{2}.$$
	This is zero if and only if $b_3\leq 1$ and $d_1+d_2-a<2$ or equivalently $-(a+1)\leq b_1+b_2$. On the other hand if $d_1+d_2<2+a$ then clearly $Y+Z\geq 0$. Hence we obtain that $\h1(\ls)=0$ if and only if $b_3\leq 1$ and $-(a+1)\leq b_1+b_2$.
	
	\textit{Case 4: $a\leq -2$ and $b_1,b_2,b_3<0$}:\\
	Similarly as before we set
	$$A=max\Big\{Y+Z,0\Big\}$$
	where
	\begin{align*}
		Y&={a+2\choose 2}-{d_1+1\choose 2}-{d_2+1\choose 2}-{d_3+1\choose 2},\\
		Z&={d_1+d_2-a\choose 2}+{d_2+d_3-a\choose 2}+{d_3+d_1-a\choose 2}-{d_1+d_2+d_3-2a-1\choose 2}.
	\end{align*}	
	First, assume that $\h1(\ls)=0.$ If $Y+Z<0$ then $A=0$ and $\h1(\ls)=0=-Y$. Hence, $Z<0$. But $Z<0$ would in particular imply that $d_1+d_2-a<d_1+d_2+d_3-2a-1$ thus $a+1<d_3$ and so $Y<0$ which gives contradiction in the case $Y+Z<0$. Now, if $Y+Z\geq 0$ then $\h1(\ls)=0=Z.$ If $Z=0$ and not all binomial coefficients in the expansion of $Z$ vanish, then in particular $d_1+d_2-a<d_1+d_2+d_3-2a-1$ and as earlier we obtain that $Y<0$ and so $Y+Z<0$ which is contradiction. Therefore all binomial coefficients in the expansion of $Z$ vanish and this implies that $d_i+d_j<a+2$ for every pair $i\neq j$. In the other direction, it's clear that these three inequalities imply $Z=0$. If we would show that $A=Y+Z$ then it would be the end since then $\h1(\ls)=Z=0.$ To get $A=Y+Z$ it is enough to show that $A>0.$ Recall that $A$ was defined as the dimension of the space of forms of degree $a$ in three variables and with zeroes of order at least $d_i$ at the points $p_i$. By the proper choice of coordinates, we can assume that $p_1=[1:0:0],\ p_2=[0:1:0]$ and $\ p_3=[0:0:1].$ Without loss of generality we can also assume that $d_1\leq d_2,d_3$. Now it is straightforward that $x^{d_2-1}y^{d_3-1}z^{a+2-d_2-d_3}$ vanish at $p_1,p_2,p_3$ with sufficient order hence $A\geq 1$. Finally from this case, we obtain the following set of inequalities which completes the analysis of this case: $$-(a+1)\leq min\{b_{i}+b_{j}\}.$$
	\par
	Notice that we considered all the possible cases up to the permutation of indices of $b_i$. Therefore combining all possible cases we ultimately get
	$$a>-2,\  b_i\leq 1\  \text{and}\ -a-1\leq min\{b_i,b_i+b_j\}$$
\end{proof}
\begin{proof}[Proof of Theorem \ref{exm23}]
	First, using Proposition \ref{mo}, we get that
	$$\h0(\cM)={\ell +3 \choose 2}-{\ell+2 \choose 2}-{2 \choose 2}=\ell +1>0.$$
	For $\cM(-h)$, we have ${\ell \choose 2}-{\ell+1 \choose 2}<0$. So
	$$\h0(\cM(-h))=0,$$
	that is, $\cM$ is initialized. Next, we have        
	\begin{align*}
		\cM(th)&=\oo_X((\ell +1+3t)m-(\ell +1+t)e_1-(t+1)e_2-te_3)\\
		&=\oo_X(a'm+b_1'e_1+b_2e_2'+b_3'e_3).
	\end{align*}
	It's clear from Lemma \ref{tlo} that $\h1 (\cM(th))=0$ for $t\geq 0.$ If $t<0$ and $\ell +1+3t>-2$ then 
	$$-(a'+1)=-\ell-2-3t>-\ell-2-2t=b_1'+b_2'=min\{b_i,b_{i'}+b_{j'}\}.$$
	Hence $\h1 (\cM(th)) >0.$ For $\ell+1+3t\leq -2$ and $t>-(\ell+1)$, we have
	$$0>-(\ell+1+t)=b_1' .$$
	So $\h1 (\cM(th))>0$. Finally, for $t\leq -(\ell+1)$, we have $\h1 (\cM(th))=0$.\\
	Therefore $\h1(\cM(th))\neq 0$ precisely for $-\ell \leq t\leq -1.$
\end{proof}

\begin{theorem}\label{exr223}
	Let $X$ be the blow up of $\p2$  at two distinct or three non collinear points. Then, there exist an indecomposable, initialized $\ell$-away ACM bundle of rank 2 on $X$.
\end{theorem}
\begin{proof}
	Consider the extension 
	\begin{equation}\label{sesexr223}
		0\to \cM \to \cE \to \ls \to 0 
	\end{equation}
	where $\cM = \oo_X((\ell +1)m-(\ell +1)e_1-e_2)$ and $\ls =\oo_X((\ell +1)m-(\ell +1)e_2-e_1)$. 
	
	We have
	$$\ext1(\ls,\cM)=\h1(\oo_X(\ell e_2 - \ell e_1 ) > 0 .$$
	So, we have a non-split extension. Since both $\cM$ and $\ls$ are line bundles, they are $\mu$-stable. Also, it is easy to see that they have the same slope. Therefore, $\cE$ is a simple bundle by \cite[Lemma 4.2]{CHGS12}. Hence, $\cE$ is indecomposable.
	
	We know that  $\cM $ is an initialized $\ell$-away ACM line bundle by Theorem \ref{exm23}. By symmetry, it is clear that $\ls$ is also an initialized $\ell$-away ACM line bundle and $\h1(\ls (th))\neq 0$ for $-\ell \leq t\leq -1$. Therefore, $\cE$ is an initialized $\ell$-away ACM bundle of rank 2.
\end{proof}

\section{Conflicts of Interest}
On behalf of all authors, the corresponding author states that there is no conflict of interest. 

\section{Acknowledgement}
We thank an anonymous referee for the careful reading of the manuscript and useful comments.

\bigskip
\noindent
Filip Gawron,\\
Faculty of Mathematics and Computer Science, Jagiellonian University,\\
ul. {\L}ojasiewicza 6,\\
30-348 Krak{\'o}w, Poland\\
e-mail: {\tt filipux.gawron@student.uj.edu.pl}

\bigskip
\noindent
Ozhan Genc,\\
Faculty of Mathematics and Computer Science, Jagiellonian University,\\
ul. {\L}ojasiewicza 6,\\
30-348 Krak{\'o}w, Poland\\
e-mail: {\tt ozhangenc@gmail.com}

\end{document}